\documentclass[reqno]{amsart}
\usepackage{CJK}
\usepackage{hyperref}
\usepackage{cite}
\usepackage{amsmath}
\usepackage{mathrsfs}
\def\dive{\operatorname{div}}

\numberwithin{equation}{section}

\newtheorem{theorem}{Theorem}[section]
\newtheorem{lemma}[theorem]{Lemma}
\newtheorem{definition}[theorem]{Definition}

\newtheorem{proposition}[theorem]{Proposition}
\newtheorem{remark}[theorem]{Remark}
\newtheorem{corollary}[theorem]{Corollary}

\allowdisplaybreaks

\begin{document}
\title[\hfil Regularity for quasi-linear parabolic equations$\dots$] {Regularity for quasi-linear parabolic equations with nonhomogeneous degeneracy or singularity}
\author[Y. Fang, C. Zhang]{Yuzhou Fang and Chao Zhang$^*$}

\thanks{$^*$Corresponding author.}

\address{Yuzhou Fang\hfill\break School of Mathematics, Harbin Institute of Technology, Harbin 150001, China}
 \email{18b912036@hit.edu.cn}

\address{Chao Zhang\hfill\break School of Mathematics and Institute for Advanced Study in Mathematics, Harbin Institute of Technology, Harbin 150001, China}
 \email{czhangmath@hit.edu.cn}

\subjclass[2010]{35B65, 35K65, 35D40, 35K92, 35K67}   \keywords {Quasi-linear parabolic equation; normalized $p$-Laplacian; nonhomogeneous degeneracy or/and singularity; viscosity solution; comparison principle; regularity.}

\maketitle

\begin{abstract}
We introduce a new class of quasi-linear parabolic equations involving nonhomogeneous degeneracy or/and singularity
$$
\partial_t u=[|D u|^q+a(x,t)|D u|^s]\left(\Delta u+(p-2)\left\langle D^2 u\frac{D u}{|D u|},\frac{D u}{|D u|}\right\rangle\right),
$$
where $1<p<\infty$, $-1<q\leq s<\infty$ and $a(x,t)\ge 0$. The motivation to investigate this model stems not only from the connections to tug-of-war like stochastic games with noise, but also from the non-standard growth problems of double phase type.  According to different values of $q,s$, such equations include nonhomogeneous degeneracy or singularity, and may involve these two features simultaneously. In particular, when $q=p-2$ and $q<s$, it will encompass the parabolic $p$-Laplacian both in divergence form and in non-divergence form. We aim to explore the from $L^\infty$ to $C^{1,\alpha}$ regularity theory for the aforementioned problem. To be precise, under some proper assumptions, we use geometrical methods to establish the local H\"{o}lder regularity of spatial gradients of viscosity solutions.
\end{abstract}

\section{Introduction}

\label{intro} \noindent

Let $B_r\subset \mathbb{R}^n$ be a ball with radius $r$ centered at the origin and $Q_r=B_r\times(-r^2,0]$. In this paper, we are concerned with the interior regularity for viscosity solutions to the following quasi-linear parabolic equation with nonhomogeneous degeneracy or/and singularity
\begin{equation}
\label{main}
\partial_t u=[|D u|^q+a(x,t)|D u|^s]\Delta_p^N u \quad \textmd{in } Q_1,
\end{equation}
where $1<p<\infty$, $-1<q\leq s<\infty$, $0\leq a(x,t)\in C^1(Q_1)$, and $\Delta_p^N$ denotes the normalized $p$-Laplace operator defined by
$$
\Delta_p^N u=\Delta u+(p-2)\left\langle D^2u\frac{D u}{|D u|},\frac{D u}{|D u|}\right\rangle=\left(\delta_{ij}+(p-2)\frac{u_i u_j}{|D u|^2}\right)u_{ij}.
$$
Here the summation convention is exploited and $D u$ is the gradient of $u$ in the spatial variable $x$. Throughout this paper, $u_i=\partial_{x_i}u$, $u_{ij}=\partial_{x_ix_j}u$, $D_{x,t}u=(\partial_tu,\partial_{x_1}u,\cdots,\partial_{x_n}u)^T$.

Over the last decade, a linkage between the stochastic tug-of-war games and nonlinear equations of $p$-Laplacian type, starting with the pioneering papers of Peres-Schramm-Sheffield-Wilson \cite{PS09} and Peres-Sheffield \cite{PS08}, has received lots of attention. For the parabolic scenario, Manfredi-Parviainen-Rossi \cite{MPR10} proved that the solutions to
\begin{equation}
\label{0-2}
\partial_tu=\Delta_p^N u
\end{equation}
could be derived as the limits of value functions for tug-of-war games with noise if the parameter controlling the size of the possible steps tends to zero. We remark that the normalized $p$-Laplacian can be regarded as the 1-homogeneous version of standard $p$-Laplacian or as a mixture of the Laplacian and normalized infinity Laplacian, $\Delta_\infty^N u=|D u|^{-2}\langle D^2uD u,D u\rangle$. The lower regularity for solutions of the homogeneous or nonhomogeneous elliptic normalized $p$-Laplace equation was obtained in \cite{LPS13,Ruo16} (see \cite{BG13,BG2015,Doe11} for the parabolic analogue). The first contribution on the $C^{1,\alpha}$-regularity for such equations is due to the seminal work of Jin-Silvestre \cite{JS17}, where they established the interior H\"{o}lder gradient estimates of solutions to \eqref{0-2}. This result was extended to the inhomogeneous parabolic normalized $p$-Laplacian in \cite{AP18}. For the inhomogeneous elliptic counterpart
$$
-\Delta_p^Nu=f(x) \quad \textmd{in } \Omega,
$$
Attouchi-Parviainen-Ruosteenoja  \cite{APR17} showed that the solutions are locally $C^{1,\alpha}$ regular under the condition that $f(x)\in L^q$ with $q\leq\infty$ possessing a suitably large lower bound; see also \cite{BM19} for the case that $f(x)\in L(n,1)$, where $L(n,1)$ denotes the standard Lorentz space. In addition, the existence of viscosity solutions to \eqref{0-2} has been obtained in \cite{BG13,BG2015,Doe11} by using approximation techniques that are different from the game-theoretic arguments \cite{MPR10}. Juutinen \cite{Juu14} investigated the asymptotic behavior for \eqref{0-2}. For more results on the stochastic tug-of-war game and the $p$-Laplacian operators, see for instance \cite{KKK14,LM14,MPR12,Rossi11}.

On the other hand, equation \eqref{main} is motivated by the double phase problems as well. We observe that equation \eqref{main} is a new model of quasi-linear parabolic equations featuring a nonhomogeneous degenerate or/and singular term modelled on the double phase integrand
\begin{equation}
\label{0-3}
H(x,t,\xi):=|\xi|^p+a(x,t)|\xi|^q, \quad\quad a(x,t)\geq 0,\quad 1<p\leq q.
\end{equation}
In the elliptic case (i.e., the function $a$ is independent of $t$), from a variational point of view, \eqref{0-3} is closely related to the following energy functional
\begin{equation}
\label{0-4}
u\mapsto\int (|D u|^p+a(x)|D u|^q)\,dx,
\end{equation}
which was originally introduced by Zhikov \cite{Zhi93,Zhi95}  in the context of homogenization and Lavrentiev phenomenon. Such functionals can provide useful models for describing the behaviours of strongly anisotropic materials. More precisely, considering two diverse materials with hardening exponents $p$ and $q$ separately, the modulating coefficient $a(\cdot)$ determines the geometry of the mixture composed of the two, according to whether $x$ belongs to the zero set $\{a(x)=0\}$ or not. These functionals with non-standard growth conditions
$$
u\mapsto\int_\Omega F(x,u,D u)\,dx, \quad \nu|\xi|^p\leq F(x,u,\xi)\leq L(|\xi|^q+1),
$$
have been a surge of interest. In the autonomous case that energy density $F(x,u,D u)\equiv F(D u)$, the regularity theory for minima of such functionals is by now well-known from the prominent works of Marcellini \cite{Mar89,Mar91,Mar96}. The investigation of double phase functional \eqref{0-4} has been continued in a series of nice papers by Colombo, Mingione et al. For instance, the local $C^{1,\alpha}$-regularity for minimizers of functional \eqref{0-4} was obtained in \cite{BCM18,CM15,CM215} under some hypotheses that $a(x)$ is H\"{o}lder continuous and the magnitude of the difference $q-p\geq 0$ is suitably small. Whereafter, the Calder\'{o}n-Zygmund type estimates for the weak solutions to
$$
\dive(|D u|^{p-2}D u+a(x)|D u|^{q-2}D u)=\dive(|F|^{p-2}F+a(x)|F|^{q-2}F)
$$
were proved in \cite{BO17,CM16,DeFM}. See also \cite{DeFM202} for the manifold constrained problem, \cite{CDeF20} for the obstacle problem, \cite{FZ20} for the equivalence of weak and viscosity solutions and \cite{CZ20} for the potential theory. More results can be found in \cite{BBO20,CS16,FVZZ20,LD18} and reference therein. Very recently, De Filippis \cite{DeF20} established the quantitative gradient bounds for weak solutions to the following parabolic double phase equations
$$
\partial_tu-\dive(|D u|^{p-2}D u+a(x,t)|D u|^{q-2}D u)=0.
$$
The Harnack's type inequality for this equation  was also derived in \cite{BS20}. However, as far as we know, the relevant regularity results regarding the parabolic double phase models are vary rare in the literature.

Influenced by the aforementioned works, we in the present paper introduce a new model \eqref{main} by combining the normalized $p$-Laplacian operator with the double phase gradient-diffusion. It is worthwhile mentioning that problem \eqref{main} exhibits some novel and intriguing characteristics. In the cases that $0<q\leq s$ and $-1<q\leq s<0$, \eqref{main} possesses the nonhomogeneous degeneracy and singularity, respectively. Furthermore, \eqref{main} has both singularity and degeneracy when $-1<q<0<s$. In particular, equation \eqref{main} incorporates $p$-Laplacian in divergence form together with $p$-Laplacian in non-divergence form simultaneously provided that $q=p-2$ and $s>q$, that is,
$$
\partial_tu-\dive(|D u|^{p-2}D u)-a(x,t)|D u|^s\Delta_p^N u=0.
$$
Meanwhile, equation \eqref{main} is a natural extension of canonical quasi-linear parabolic equations with singularity or degeneracy, whose highly celebrated prototype is
\begin{equation}
\label{0-5}
\partial_tu-|D u|^q\Delta_p^Nu=0.
\end{equation}
Imbert-Jin-Silvestre \cite{IJS19} showed the interior $C^{1,\alpha}$-regularity of viscosity solutions $u$ to \eqref{0-5} in $Q_1$, which states that
$$
\|D u\|_{C^\alpha(Q_{1/2})}\leq C
$$
and
$$
\sup_{\stackrel{(x,t),(x,s)\in Q_{1/2}}{t \neq s}}\frac{|u(x,t)-u(x,s)|}{|t-s|^\frac{1+\alpha}{2-\alpha q}}\leq C.
$$
Later, for the nonhomogeneous analogue,
\begin{equation*}
\partial_tu-|D u|^q\Delta_p^Nu=f(x,t),
\end{equation*}
the local $C^{1,\alpha}$-regularity of solutions was completed under the assumption that $f$ is continuous and bounded; see \cite{Att20} for the degenerate case $q\geq 0$ and \cite{AR20} for the singular case $-1<q<0$. Additionally, several extra aspects of such equations have  already been explored as well, such as existence and uniqueness of solutions \cite{CGG91,Dem11}, the comparison principles \cite{GGIS91,OS97}, Aleksandrov-Bakelman-Pucci type estimate \cite{ACP11}, parabolic Harnack's inequality \cite{PV}. For the related regularity results in the elliptic context, we refer to \cite{AR18,BD10,IS13} and the references therein.

In this work we make use of a unified geometrical method developed by Jin-Silvestre \cite{JS17} and Imbert-Jin-Silvestre\cite{IJS19} to study the interior H\"{o}lder continuity for the spatial gradient of solutions to \eqref{main}. The strategies of the proof concentrate mainly on verifying that the oscillation of gradient is declining in a shrinking sequence of parabolic cylinders, and then reducing the iterative step to a dichotomy between two cases: either the gradient $D u$ stays close to a fixed unit vector $e$ for most points $(x,t)$ (in measure), or it does not, and at last patching these two alternatives together. We shall first prove the $C^{1,\alpha}$-regularity for solutions to \eqref{main} with the strong restriction that $\|D_{x,t}a(x,t)\|_{L^\infty(Q_1)}$ is small (less than 1). Then, by employing a scaling technique, we infer the $C^{1,\alpha}$ estimates for solutions of \eqref{main} under the assumption that $a(x,t)\in C^1(Q_1)$.  Due to the presence of coefficient $a(x,t)$ and the fact that the nonhomogeneous $(q,s)$-growth gradient-diffusion terms are intertwined  in equation \eqref{main}, the theoretical analysis in the current study is radically much more challenging than the previous ones. The significant distinctions and difficulties are as follows.  First, in order to obtain the improvement of oscillation for $|D u|$, it is indispensable to incorporate more terms involving $D a(x,t)$ when we differentiate the regularized equation \eqref{2-1} in $x$-variable. Therefore, in comparison to the proof of Lemma 4.1 in \cite{IJS19}, we need additional elaborate analyses and construct a  much more complicated auxiliary function. Second, the comparison principle (Proposition \ref{pro2-4}) cannot plainly follow from the known results due to the presence of $a(x,t)$. We have to meticulously apply the information from the maximum principle for semicontinuous functions, together with properties such as the local Lipschitz continuity with respect to the matrix square root. Moreover, the Lipschitz regularity of viscosity solutions to \eqref{main} plays a rather crucial role in the proof. To the best of our knowledge, the comparison principle is new, which is also of independent interest. Finally, in order to establish the Lipschitz estimates for equation \eqref{2-1} in the spatial variable, we shall employ Ishii-Lions' method twice: we first use it to deduce the solutions are H\"{o}lder continuity in $x$-variable, and then we rely on this H\"{o}lder regularity and use the Ishii-Lions' method again to show the Lipschitz estimates.


The paper is organized as follows. In Section \ref{sec1}, we first recall the definition of viscosity solutions to \eqref{main} and then give the main result derived in this paper. Section \ref{sec2} contains the Lipschitz continuity in the spatial variables, the H\"{o}lder continuity in the time variable as well as two important properties of viscosity solutions. Section \ref{sec3} is devoted to establishing the H\"{o}lder estimates on the spatial gradients of solutions, which is the most technically challenging part. We complete the proof of comparison principle (Proposition \ref{pro2-4}) in Section \ref{sec4}. The technical proof of the Lipschitz continuity in the spatial variables (Lemma \ref{lem2-1}) is postponed to Section \ref{sec5}. At last, we in Section \ref{sec6} present the proof of the boundary estimates, Proposition \ref{pro3-7}.

\section{Main result}
\label{sec1}

The aim of this paper is to establish the interior H\"{o}lder estimates for spatial gradients of solutions to problem \eqref{main}. To this end, the following hypotheses will be in force. We first assume that
\begin{equation}
\label{1-1}
1<p<\infty.
\end{equation}
Furthermore, concerning the nonhomogeneous degeneracy or/and singularity term appearing in \eqref{main}, we shall require that the exponents $q, s$ fulfill
\begin{equation}
\label{1-2}
-1<q\leq s<\infty,
\end{equation}
and that the modulating coefficient $a(\cdot)$ is such that
\begin{equation}
\label{1-3}
0<a^-:=\inf_{Q_1}a(x,t)\leq a(x,t)\leq a^+:=\sup_{Q_1}a(x,t)<\infty,
\end{equation}
and
\begin{equation}
\label{1-4}
a(x,t)\in C^1(Q_1) \quad \text{and} \quad  A:=\|D_{x,t}a(x,t)\|_{L^\infty(Q_1)}<\infty.
\end{equation}
Throughout this article, the assumptions \eqref{1-1} and \eqref{1-2}  are always supposed to hold.


In the degenerate case (i.e., $q\geq0$), the definition of viscosity solutions is straightforward. Nonetheless, the formulations of defining viscosity solutions cannot be displayed specifically for the case when singularity occurs (i.e., $-1<q<0$). Hence we here adopt the same notion of viscosity solutions as the one utilized in \cite{JLM01} to provide a unified way of defining solutions in the degenerate and singular conditions. Next let us recall the definition of viscosity solutions to \eqref{main}.

\begin{definition}[viscosity solution]
\label{def1}
A finite almost everywhere and lower semicontinuous function $u:Q_1\rightarrow \mathbb{R}\cup\{+\infty\}$ is a viscosity supersolution to \eqref{main} in $Q_1$, if whenever $(x_0,t_0)\in Q_1$ and $\varphi\in C^2(Q_1)$ are such that $u-\varphi$ attains a local minimum at $(x_0,t_0)$ and moreover $D\varphi(x,t)\neq0$ for $x\neq x_0$, then we obtain
\begin{equation*}
\limsup_{\stackrel{(x,t)\rightarrow(x_{0},t_{0})}{x\neq x_{0}}}\left(\partial_t \varphi(x,t)-[|D\varphi(x,t)|^q+a(x,t)|D\varphi(x,t)|^s]\Delta_p^N\varphi(x,t)\right)\geq0.
\end{equation*}
A finite almost everywhere and upper semicontinuous function $u:Q_1\rightarrow \mathbb{R}\cup\{-\infty\}$ is a viscosity subsolution to \eqref{main} in $Q_1$, if whenever $(x_0,t_0)\in Q_1$ and $\varphi\in C^2(Q_1)$ are such that $u-\varphi$ reaches a local maximum at $(x_0,t_0)$ and moreover $D\varphi(x,t)\neq0$ for $x\neq x_0$, then we derive
\begin{equation*}
\liminf_{\stackrel{(x,t)\rightarrow(x_{0},t_{0})}{x\neq x_{0}}}\left(\partial_t \varphi(x,t)-[|D\varphi(x,t)|^q+a(x,t)|D\varphi(x,t)|^s]\Delta_p^N\varphi(x,t)\right)\leq0.
\end{equation*}

A function $u$ is called a viscosity solution to \eqref{main} if and only if it is both viscosity super- and subsolution.
\end{definition}

\begin{remark}
When $D\varphi(x_0,t_0)\neq0$, these limits above are explicit,
$$
\partial_t \varphi(x_0,t_0)-[|D\varphi(x_0,t_0)|^q+a(x_0,t_0)|D\varphi(x_0,t_0)|^s]\Delta_p^N\varphi(x_0,t_0)\geq(\leq)0.
$$
In addition, when $0\leq q\leq s$, the condition that $D\varphi(x,t)\neq0$ for $x\neq x_0$ can be removed actually. For example, if $D\varphi(x_0,t_0)=0$, after careful computations, the supremum limit turns into
\begin{itemize}
\item  [(i)] the case that $q=s=0$,
\begin{equation*}
\begin{cases}
\partial_t\varphi(x_0,t_0)-(1+a(x_0,t_0))\left(\mathrm{tr}(D^2\varphi(x_0,t_0))+(p-2)\lambda_{\rm min}(D^2\varphi(x_0,t_0))\right)\geq0,  & \text{\textmd{}} p\geq2, \\[2mm]
\partial_t\varphi(x_0,t_0)-(1+a(x_0,t_0))\left(\mathrm{tr}(D^2\varphi(x_0,t_0))+(p-2)\lambda_{\rm max}(D^2\varphi(x_0,t_0))\right)\geq0,  & \text{\textmd{}} p\in(1,2).
\end{cases}
\end{equation*}

 \item  [(ii)]  the case that $0=q<s$,
\begin{equation*}
\begin{cases}
\partial_t\varphi(x_0,t_0)-\left(\mathrm{tr}(D^2\varphi(x_0,t_0))+(p-2)\lambda_{\rm min}(D^2\varphi(x_0,t_0))\right)\geq0,  & \text{\textmd{when}}\quad p\geq2, \\[2mm]
\partial_t\varphi(x_0,t_0)-\left(\mathrm{tr}(D^2\varphi(x_0,t_0))+(p-2)\lambda_{\rm max}(D^2\varphi(x_0,t_0))\right)\geq0,  & \text{\textmd{when}}\quad p\in(1,2).
\end{cases}
\end{equation*}

\item  [(iii)]  the case that $0<q\leq s$,
$$
\partial_t \varphi(x_0,t_0)\geq0.
$$
\end{itemize}
Here $\mathrm{tr}(N)$ is the trace of matrix $N$, and $\lambda_{\rm min}(N)$ ($\lambda_{\rm max}(N)$) denotes the minimum (maximum) eigenvalue of $N$.

The infimal limit in definition can be tackled analogously.
\end{remark}

Now we are in position to state our main contribution of this work.

\begin{theorem}
\label{thm0}
Let the conditions \eqref{1-1}--\eqref{1-4} be in force. Suppose that $u$ is a bounded viscosity solution to equation \eqref{main} in $Q_1$. Then there are two constants $\alpha\in(0,1)$ and $C>0$, both depending upon $n,p,q,s,a^-,a^+,A$ and $\|u\|_{L^\infty(Q_1)}$, such that the following estimates hold
$$
\|D u\|_{C^\alpha(Q_{1/2})}\leq C
$$
and
$$
\sup_{\stackrel{(x,t),(x,s)\in Q_{1/2}}{t \neq s}}\frac{|u(x,t)-u(x,s)|}{|t-s|^\frac{1+\alpha}{2-\alpha q}}\leq C.
$$
\end{theorem}

\section{Lower regularity for solutions}
\label{sec2}

In order to circumvent some technical difficulties created by the lack of smoothness of viscosity solutions to \eqref{main}, we first study the regularized equation below
\begin{equation}
\label{2-1}
\partial_t u=\big[(|D u|^2+\varepsilon^2)^\frac{q}{2}+a(x,t)(|D u|^2+\varepsilon^2)^\frac{s}{2}\big]\left(\delta_{ij}+(p-2)\frac{u_i u_j}{|D u|^2+\varepsilon^2}\right)u_{ij}
\end{equation}
in $Q_1$, where $0<\varepsilon<1$. Then we devote to obtaining uniform estimates with respect to $\varepsilon$ so that we could pass to the limit in the end.

In this section, we are going to show the Lipschitz continuity in the spatial variables and the H\"{o}lder continuity in the time variable. Now we first present the Lipschitz estimates independent of $\varepsilon$ on solutions to equation \eqref{2-1}. However, the proof of this lemma is rather long and delicate, which was postponed to Section \ref{sec5}. It is worth mentioning that, from the proof below, we can easily find that the Lipschitz estimates also hold true for $\varepsilon=0$. We state this result as follows.

\begin{lemma}[Local Lipschitz estimates in $x$-variable]
\label{lem2-1}
Let $\varepsilon\in[0,1)$ and the assumptions \eqref{1-1} and \eqref{1-2} be in force. Let $u$ be a smooth solution to \eqref{2-1} in $Q_1$. Assume that $a(x,t)\geq a^->0$ and $a(x,t)$ is uniformly Lipschitz continuous in $x$-variable, that is, there exists a constant $C_{\rm lip}>0$, independent of $t$-variable, such that $|a(x,t)-a(y,t)|\leq C_{\rm lip}|x-y|$. Then for all $r \in(0,\frac{7}{8}]$, there holds that
\begin{equation*}
|u(x,t)-u(y,t)|\leq C|x-y|
\end{equation*}
for $(x,t), (y,t)\in \overline{Q_r}$, where $C>0$ depends on $n,p,q,s, a^-,C_{\rm lip}$ and $\|u\|_{L^\infty(Q_1)}$.
\end{lemma}

\begin{remark}
From the proof of this lemma in Section \ref{sec5}, we can see the explicit dependencies of the above constant $C$ with
$$
C:=C(n,p,q,s)\|u\|_{L^\infty(Q_1)}\left[1+\left(\frac{C_{\rm lip}}{a^-}\right)^2\right].
$$
If $a(x,t)$ is supposed to be of class $C^1(\overline{Q_1})$ in the previous lemma, then the constant $C_{\rm lip}>0$ can be replaced by $A:=\|D_{x,t}a(x,t)\|_{L^\infty(Q_1)}$.
\end{remark}

Based on the Lipschitz estimates above and a simple comparison argument, we can demonstrate that the solutions to \eqref{2-1} are H\"{o}lder continuous in $t$, which will be utilized in Lemma \ref{lem3-4} below.

\begin{lemma}[Local H\"{o}lder estimates in $t$-variable]
\label{lem2-2}
Suppose that $u$ is a smooth solution of \eqref{2-1} in $Q_1$ with $0<\varepsilon<1$. Let $a(x,t)$ satisfy that $|a(x,t)-a(y,t)|\leq C_{\rm lip}|x-y|$ in $Q_1$. Then under the assumptions \eqref{1-1}--\eqref{1-3}, the following estimates hold:

\begin{itemize}
\item[(i)] for $0\leq q\leq s$,
  $$
  \sup_{\stackrel{(x,t),(x,s)\in Q_{3/4}}{t \neq s}}\frac{|u(x,t)-u(x,s)|}{|t-s|^\frac{1}{2}}\leq C;
  $$

\smallskip

\item[(ii)] for $-1<q\leq s<0$ or $-1<q<0\leq s$,
  $$
   \sup_{\stackrel{(x,t),(x,s)\in Q_{3/4}}{t \neq s}}\frac{|u(x,t)-u(x,s)|}{|t-s|^\frac{1}{\beta(1+s)-s}}\leq C,
  $$
  where $\beta=\frac{q+2}{q+1}$ and $C$ depends on $n,p,q,s,a^-,a^+,C_{\rm lip}$ and $\|u\|_{L^\infty(Q_1)}$.
\end{itemize}
\end{lemma}

\begin{proof}
Let $\beta\geq2$ be determined later, according to three different scenarios, i.e., $-1<q\leq s<0$, $-1<q<0\leq s$ and $0\leq q\leq s$. For all $t_0\in \left[-\left(\frac{3}{4}\right)^2,0\right)$ and $\eta>0$, we now assert that there are two constants $L_1,L_2>0$ such that
\begin{equation}
\label{2-2}
u(x,t)-u(0,t_0)\leq \eta+L_1(t-t_0)+L_2|x|^\beta=:\varphi(x,t)
\end{equation}
for any $\overline{B_{3/4}}\times[t_0,0]$. We first select $L_2\geq 2(\frac{4}{3})^\beta\|u\|_{L^\infty(Q_1)}$ such that \eqref{2-2} holds for $x\in \partial B_{3/4}$, and in turn take $L_2$ such that \eqref{2-2} holds for $t=t_0$. That is to say, we can choose properly such $L_2>0$ that \eqref{2-2} does hold on the boundary of $\overline{B_{3/4}}\times[t_0,0]$. Indeed, due to $u$ is Lipschitz continuous in the spatial variables, we may take
$$
L_2\geq \frac{\|D u\|^\beta_{L^\infty(Q_{7/8})}}{\eta^{\beta-1}}
$$
to guarantee that
$$
\eta+L_2|x|^\beta\geq\|D u\|_{L^\infty(Q_{7/8})}|x|
$$
by Young's inequality, which implies that \eqref{2-2} is true for $t=t_0$. Here we note that $\|D u\|_{L^\infty(Q_{7/8})}$ is bounded depending on $n,p,q,s,a^-,C_{\rm lip}$ and $\|u\|_{L^\infty(Q_1)}$. In the rest of proof, we fix
$$
L_2=\eta^{1-\beta}\|D u\|^\beta_{L^\infty(Q_{7/8})}+2\left(\frac{4}{3}\right)^\beta\|u\|_{L^\infty(Q_1)}+1.
$$
Next, we are ready to select $L_1$ such that $\varphi(x,t)$ is a supersolution to certain equations. Inequality \eqref{2-2} then follows by the comparison principle. The remaining proof is completed under three diverse cases.

\smallskip

\textbf{Case 1.} $-1<q\leq s<0$. We shall show that $\varphi(x,t)$ is a supersolution to equation \eqref{2-1}, that is,
\begin{equation}
\label{2-3}
\partial_t\varphi-\big[(|D\varphi|^2+\varepsilon^2)^\frac{q}{2}+a(x,t)(|D\varphi|^2+\varepsilon^2)^\frac{s}{2}\big]
\left(\delta_{ij}+(p-2)\frac{\varphi_i\varphi_j}{|D\varphi|^2+\varepsilon^2}\right)\varphi_{ij}\geq0.
\end{equation}
We first calculate
$$
D\varphi=\beta L_2|x|^{\beta-2}x,
$$
$$
D^2\varphi=\beta L_2|x|^{\beta-2}I+\beta(\beta-2)L_2|x|^{\beta-2}\frac{x}{|x|}\otimes\frac{x}{|x|},
$$
where it is easy to see that $D^2\varphi$ is a positive definite matrix and
\begin{equation}
\label{2-3-1}
\|D^2\varphi\|\leq \beta(\beta-1)L_2|x|^{\beta-2}.
\end{equation}
Here $\xi\otimes\xi$ is the matrix with entries $\xi_i\xi_j$ for a vector $\xi\in\mathbb{R}^n$. We next evaluate
\begin{align*}
&\quad\big[(|D\varphi|^2+\varepsilon^2)^\frac{q}{2}+a(x,t)(|D\varphi|^2+\varepsilon^2)^\frac{s}{2}\big]
\left(\delta_{ij}+(p-2)\frac{\varphi_i\varphi_j}{|D\varphi|^2+\varepsilon^2}\right)\varphi_{ij}\\
&\leq C(n,p)(|D\varphi|^q+a^+|D\varphi|^s)\|D^2\varphi\|\\
&\leq C\left(L^{1+q}_2|x|^{q(\beta-1)+\beta-2}+L^{1+s}_2|x|^{s(\beta-1)+\beta-2}\right)\\
&\leq C(L^{1+q}_2+L^{1+s}_2)
\end{align*}
by taking $\beta\geq\frac{q+2}{q+1}(>2)$, where $C$ depends only on $n,p,q,s,a^+$. Thereby, in order to assure \eqref{2-3}, we need to fix $L_1=C(L^{1+q}_2+L^{1+s}_2)$.

Finally, applying the comparison principle together with the choices of $L_1,L_2$, we arrive at
\begin{align*}
&\quad u(0,t)-u(0,t_0)\\
&\leq\eta+L_1(t-t_0)\\
&\leq\eta+C\left(\eta^{1-\beta}\|D u\|^\beta_{L^\infty(Q_{7/8})}+2(4/3)^\beta\|u\|_{L^\infty(Q_1)}+1\right)^{1+q}(t-t_0)\\
&\quad+C\left(\eta^{1-\beta}\|D u\|^\beta_{L^\infty(Q_{7/8})}+2(4/3)^\beta\|u\|_{L^\infty(Q_1)}+1\right)^{1+s}(t-t_0)\\
&\leq\eta+C\|D u\|^{\beta(1+q)}_{L^\infty(Q_{7/8})}\eta^{(1-\beta)(1+q)}|t-t_0|+C\|D u\|^{\beta(1+s)}_{L^\infty(Q_{7/8})}\eta^{(1-\beta)(1+s)}|t-t_0|\\
&\quad+C(\|u\|_{L^\infty(Q_1)}+1)^{1+s}|t-t_0|.
\end{align*}
We now pick $\eta=|t-t_0|^\gamma$ with $0<\gamma<1$ to be fixed later. Then it follows that
\begin{align*}
&\quad u(0,t)-u(0,t_0)\\
&\leq |t-t_0|^\gamma+C\|D u\|^{\beta(1+q)}_{L^\infty(Q_{7/8})}|t-t_0|^{\gamma(1-\beta)(1+q)+1}\\
&\quad+C\|D u\|^{\beta(1+s)}_{L^\infty(Q_{7/8})}|t-t_0|^{\gamma(1-\beta)(1+s)+1}
+C(\|u\|_{L^\infty(Q_1)}+1)^{1+s}|t-t_0|.
\end{align*}
Then $\gamma$ can be chosen as
\begin{equation*}
\begin{cases}
\gamma(1-\beta)(1+q)+1-\gamma\geq0,\\[2mm]
\gamma(1-\beta)(1+s)+1-\gamma\geq0, \\[2mm]
0<\gamma<1,
\end{cases}
\end{equation*}
which leads to
$$
\gamma\leq\frac{1}{\beta(1+s)-s}(<1)
$$
by noting that $-1<q\leq s<0$. As has been stated above, we could determine $\beta=\frac{q+2}{q+1}$ and $\gamma=\frac{1}{\beta(1+s)-s}$. We then get the desired result for the case that  $-1<q\leq s<0$.

\smallskip

\textbf{Case 2.} $-1<q<0\leq s$. Similarly to Case 1, we get
\begin{align*}
&\quad\big[(|D\varphi|^2+\varepsilon^2)^\frac{q}{2}+a(x,t)(|D\varphi|^2+\varepsilon^2)^\frac{s}{2}\big]
\left(\delta_{ij}+(p-2)\frac{\varphi_i\varphi_j}{|D\varphi|^2+\varepsilon^2}\right)\varphi_{ij}\\
&\leq C(n,p)(|D\varphi|^q+a^+(|D\varphi|^s+1))\|D^2\varphi\|\\
&\leq C\left(L^{1+q}_2|x|^{q(\beta-1)+\beta-2}+L_2|x|^{\beta-2}+L^{1+s}_2|x|^{s(\beta-1)+\beta-2}\right)\\
&\leq C(L^{1+q}_2+L_2+L^{1+s}_2)
\end{align*}
by taking $\beta\geq\frac{q+2}{q+1}$, where $C$ depends only on $n,p,q,s$ and $a^+$. Hence we can choose $L_1=C(L^{1+q}_2+L_2+L^{1+s}_2)$ to ensure \eqref{2-3}.

Then through the comparison principle and the choices of $L_1,L_2$, it yields that
\begin{align*}
&\quad u(0,t)-u(0,t_0)\\
&\leq\eta+L_1(t-t_0)\\
&\leq\eta+C\|D u\|^{\beta(1+q)}_{L^\infty(Q_{7/8})}\eta^{(1-\beta)(1+q)}|t-t_0|+C\|D u\|^\beta_{L^\infty(Q_{7/8})}\eta^{1-\beta}|t-t_0|\\
&\quad+C\|D u\|^{\beta(1+s)}_{L^\infty(Q_{7/8})}\eta^{(1-\beta)(1+s)}|t-t_0|
+C(\|u\|_{L^\infty(Q_1)}+1)^{1+s}|t-t_0|.
\end{align*}
Let $\eta=|t-t_0|^\gamma$. The above display then becomes
\begin{align*}
&\quad u(0,t)-u(0,t_0)\\
&\leq|t-t_0|^\gamma+C\|D u\|^{\beta(1+q)}_{L^\infty(Q_{7/8})}|t-t_0|^{\gamma(1-\beta)(1+q)+1}+C\|D u\|^\beta_{L^\infty(Q_{7/8})}|t-t_0|^{\gamma(1-\beta)+1}\\
&\quad+C\|D u\|^{\beta(1+s)}_{L^\infty(Q_{7/8})}|t-t_0|^{\gamma(1-\beta)(1+s)+1}
+C(\|u\|_{L^\infty(Q_1)}+1)^{1+s}|t-t_0|.
\end{align*}
We shall pick $0<\gamma<1$ such that
\begin{equation*}
\begin{cases}
\gamma(1-\beta)(1+q)+1-\gamma\geq0,\\[2mm]
\gamma(1-\beta)+1-\gamma\geq0,\\[2mm]
\gamma(1-\beta)(1+s)+1-\gamma\geq0,
\end{cases}
\end{equation*}
i.e.,
$$
\gamma\leq\frac{1}{\beta(1+s)-s}.
$$
Consequently, we can fix $\beta=\frac{q+2}{q+1}$ and $\gamma=\frac{1}{\beta(1+s)-s}$. We then finish the proof for the case that  $-1<q<0\leq s$.

\smallskip

\textbf{Case 3.} $0\leq q\leq s$. This time, we can verify that $\varphi(x,t)$ is a supersolution of a linear parabolic equation with coefficients depending on $u$. That is,
$$
\partial_t\varphi-\big[(|D u|^2+\varepsilon^2)^\frac{q}{2}+a(x,t)(|D u|^2+\varepsilon^2)^\frac{s}{2}\big]
\left(\delta_{ij}+(p-2)\frac{u_iu_j}{|D u|^2+\varepsilon^2}\right)\varphi_{ij}\geq0.
$$
Because $q\geq0$ and $|D u|$ is known to be bounded by the Lipschitz continuity, we can rewrite this display as
\begin{equation}
\label{2-4}
\partial_t\varphi-a_{ij}(x,t)\varphi_{ij}\geq0,
\end{equation}
where
\begin{equation}
\label{2-5}
|a_{ij}(x,t)|\leq C(p)((\|D u\|_{L^\infty(Q_{7/8})}+\varepsilon)^q+a^+(\|D u\|_{L^\infty(Q_{7/8})}+\varepsilon)^s).
\end{equation}
Here the boundedness on $|a_{ij}(x,t)|$ depends on $n,p,q,s, a^-,a^+,C_{\rm lip}$ and $\|u\|_{L^\infty(Q_1)}$ actually. We fix $\beta=2$. Thus by \eqref{2-3-1} and \eqref{2-5}, we take
$$
L_1=C(\|D u\|_{L^\infty(Q_{7/8})}+1)^sL_2,
$$
which makes \eqref{2-4} hold true. In turn, utilizing the comparison principle again, we have
$$
u(0,t)-u(0,t_0)\leq \eta+C(\|D u\|_{L^\infty(Q_{7/8})}+1)^s(\eta^{-1}\|D u\|^2_{L^\infty(Q_{7/8})}+\|u\|_{L^\infty(Q_1)}+1)|t-t_0|.
$$
Taking
$$
\eta=(\|D u\|_{L^\infty(Q_{7/8})}+1)^{\frac{s}{2}+1}|t-t_0|^\frac{1}{2},
$$
we derive
\begin{align*}
u(0,t)-u(0,t_0)&\leq (\|D u\|_{L^\infty(Q_{7/8})}+1)^{\frac{s}{2}+1}|t-t_0|^\frac{1}{2}\\
&\quad+C(\|D u\|_{L^\infty(Q_{7/8})}+1)^s(\|u\|_{L^\infty(Q_1)}+1)|t-t_0|.
\end{align*}
This proof now is finished.
\end{proof}

Next we end this section by presenting two important properties of viscosity solutions, comparison principle and stability, which will be exploited in the proof of Theorem \ref{thm3-9} below. However, their proof shall be postponed to Section \ref{sec4}.

\begin{proposition}[Stability]
\label{pro2-3}
Assume that $\{u_i\}$ is a sequence of viscosity solutions to \eqref{2-1} in $Q_1$ with $\varepsilon_i\geq0$ such that $\varepsilon_i\rightarrow0$. Let $u_i$ converge to $u$ locally uniformly in $Q_1$. Then we can infer that $u$ is a viscosity solution to \eqref{main} in $Q_1$.
\end{proposition}

Once deriving the Lipschitz continuity of solutions to equation \eqref{main}, we can show the following comparison principle that is interesting by itself.

\begin{proposition}[Comparison principle]
\label{pro2-4}
Let the function $a(x,t)>0$ in \eqref{main} be Lipschitz continuous in time-space variables. Assume that $u$ and $v$ are a viscosity subsolution and a locally uniformly Lipschitz continuous viscosity supersolution in $x$-variable to \eqref{main} in $Q_1$, respectively. If $u\leq v$ on $\partial_pQ_1$, then there holds that
$$
u\leq v \quad\text{in } Q_1.
$$
\end{proposition}

\section{H\"{o}lder estimates on the spatial gradients}
\label{sec3}

In this section, we assume that $a(x,t)\in C^1(Q_1)$ and $A:=\|D_{x,t}a\|_{L^\infty(Q_1)}<\infty$. Since Lemma \ref{lem2-1} states that the solutions to \eqref{2-1} in $Q_1$ have uniform interior Lipschitz estimates in $x$-variable (which is independent of $\varepsilon\in[0,1)$), we can see that
$$
\|D u\|_{L^\infty(Q_{7/8})}\leq C(n,p,q,s,a^-,A,\|u\|_{L^\infty(Q_1)}).
$$
Here  we refer to Lemma \ref{lem2-1} for the explicit dependencies of parameters. In what follows, we may suppose that $D u$ is bounded in $Q_1$ for convenience, as we could obtain these conclusions in the preceding section in a larger domain such as $Q_2$.

We are going to establish the H\"{o}lder estimates on $D u$ at the origin $(0,0)$, and then deduce plainly the interior H\"{o}lder continuity of $D u$ by standard translation arguments. The idea of this proof is analogous to that in \cite{IJS19,JS17}, but there exist many extra delicate difficulties caused by the coefficient $a(x,t)$ and the $(q,s)$-growth. We will consider the so-called intrinsic (re-scaled) parabolic cylinder defined as
$$
Q^\rho_r=B_r\times(-\rho^{-q}r^2,0] \quad\text{with } r,\rho>0.
$$
The same family of parabolic cylinders $Q^\rho_r$ has been utilized in \cite{DF85}. If $u$ solves \eqref{2-1} in $Q^\rho_r$ and we denote $v(x,t)=\frac{1}{\rho r}u(rx,\rho^{-q}r^2t)$ with $(x,t)\in Q_1$, then it is easy to check that
\begin{equation}
\label{3-1}
\partial_tv=\big[(|D v|^2+\overline{\varepsilon}^2)^\frac{q}{2}+\overline{a}(x,t)(|D v|^2+\overline{\varepsilon}^2)^\frac{s}{2}\big]\left(\delta_{ij}+(p-2)\frac{v_i v_j}{|D v|^2+\overline{\varepsilon}^2}\right)v_{ij}
\end{equation}
in $Q_1$, where
$$
\overline{a}(x,t)=\rho^{s-q}a(rx,\rho^{-q}r^2t) \quad \text{and} \quad \overline{\varepsilon}=\varepsilon \rho^{-1}.
$$
Particularly, it is noteworthy that if $u$ solves \eqref{main} in $Q^\rho_r$, then $v$ (defined as before) is a solution to
\begin{equation}
\label{3-1-1}
\partial_tv=[|D v|^q+\overline{a}(x,t)|D v|^s]\left(\delta_{ij}+(p-2)\frac{v_i v_j}{|D v|^2}\right)v_{ij}  \quad \text{in } Q_1.
\end{equation}
Hence when we fix $\rho\geq \|D u\|_{L^\infty(Q_1)}+1$, we know that the solutions of \eqref{3-1} or \eqref{3-1-1} satisfy $|D v|\leq1$ in $Q_1$. Therefore, in the sequel, we may suppose that the solutions to \eqref{2-1} fulfill $|D u|\leq1$ in $Q_1$. Next, we proceed with considering \eqref{2-1} and investigate the H\"{o}lder continuity of gradients of its solutions, from which we can derive the higher regularity (Theorem \ref{thm0}) for the solutions to \eqref{main} (by sending $\varepsilon\rightarrow0$). To this end, we first show the H\"{o}lder estimates on the gradients of solutions to \eqref{main} under the assumption that $\|D_{x,t}a(x,t)\|_{L^\infty(Q_1)}$ ($\leq 1$) is small. In turn, by doing a scaling work, we eventually demonstrate the H\"{o}lder regularity for the gradients of solutions to \eqref{main} under the condition that $\|D_{x,t}a(x,t)\|_{L^\infty(Q_1)}$ is finite, that is, $|D_{x,t}a(x,t)|$ exhibits a general bound in $Q_1$.

\subsection{H\"{o}lder regularity of spatial gradients in the case that $\|D_{x,t}a(x,t)\|_{L^\infty(Q_1)}$ is small} 

We may assume $\|D_{x,t}a(x,t)\|_{L^\infty(Q_1)}\leq1$. Now we are ready to verify that when the projection of $D u$ onto the direction $e\in \mathbb{S}^{n-1}$ (i.e., $|e|=1$) is away from 1 in a large portion of $Q_1$, then in a smaller cylinder the inner product $D u\cdot e$ has improved oscillation.

\begin{lemma}
\label{lem3-1}
Let the conditions \eqref{1-1} and \eqref{1-2} be in force. Assume that $u$ is a smooth solution of \eqref{2-1} with $\varepsilon\in(0,1)$ such that $|D u|\leq1$ in $Q_1$. For each $l\in(\frac{1}{2},1)$ and $\mu>0$, if $0\leq a(x,t)\in C^1(Q_1)$ and $\|D a\|_{L^\infty(Q_1)}\leq \kappa$, where $\kappa\in(0,1]$ is a sufficiently small constant depending on $n,p,q,s,a^+, \mu$ and $l$, then there is $\tau_0\in(0,\frac{1}{4})$ only depending on $n,\mu$ and there are $\tau,\delta>0$ depending upon $n,p,q,s,a^+,\mu$ and $l$ such that for arbitrary $e\in \mathbb{S}^{n-1}$ if
$$
|\{(x,t)\in Q_1: D u\cdot e\leq l\}|>\mu|Q_1|,
$$
one has
$$
D u\cdot e<1-\delta \quad \text{in }  Q^{1-\delta}_\tau
$$
with $Q^{1-\delta}_\tau\subset Q_{\tau_0}$.
\end{lemma}

\begin{proof}
Set
\begin{equation}
\label{3-2}
a_{ij}(x,t,\eta)=\big[(|\eta|^2+\varepsilon^2)^\frac{q}{2}+a(x,t)(|\eta|^2+\varepsilon^2)^\frac{s}{2}\big]\left(\delta_{ij}+(p-2)\frac{\eta_i \eta_j}{|\eta|^2+\varepsilon^2}\right)
\end{equation}
with $\eta\in\mathbb{R}^n$ and
$$
a_{ij,m}(x,t,\eta):=\frac{\partial a_{ij}(x,t,\eta)}{\partial\eta_m},
$$
where $\eta_i$ denotes the $i$-th component of $\eta$. By differentiating equation \eqref{2-1} in $x_k$, we get
\begin{align*}
\partial_t(u_k)&=a_{ij}(x,t,D u)(u_k)_{ij}+a_{ij,m}(x,t,D u)u_{ij}(u_k)_m\\
&\quad+\partial_ka(x,t)(|D u|^2+\varepsilon^2)^{\frac{s}{2}}\left(\delta_{ij}+(p-2)\frac{u_i u_j}{|D u|^2+\varepsilon^2}\right)u_{ij}.
\end{align*}
In the rest of proof, let
\begin{equation*}
b_{ij}(D u)=(|D u|^2+\varepsilon^2)^{\frac{s}{2}}\left(\delta_{ij}+(p-2)\frac{u_i u_j}{|D u|^2+\varepsilon^2}\right).
\end{equation*}
We further have
\begin{align*}
\partial_t(D u\cdot e-l)&=a_{ij}(x,t,D u)(D u\cdot e-l)_{ij}
+a_{ij,m}(x,t,D u)u_{ij}(D u\cdot e-l)_m\\
&\quad+D a\cdot eb_{ij}(D u) u_{ij}.
\end{align*}
Let $h=|D u|^2$. Then
$$
\partial_t h=2D u\cdot D u_t, \quad h_i=2D u\cdot D u_i,
$$
$$
h_{ij}=2D u_j\cdot D u_i+2D u\cdot D u_{ij}.
$$
By direct calculation, it yields that
\begin{align*}
\partial_t h&=a_{ij}(x,t,D u)h_{ij}+a_{ij,m}(x,t,D u)u_{ij}h_m+2D a\cdot D u b_{ij}(D u)u_{ij}\\
&\quad-2a_{ij}(x,t,D u)u_{ki}u_{kj}.
\end{align*}
For $\rho=\frac{l}{4}$, define
$$
w=(D u\cdot e-l+\rho|D u|^2)_+
$$
with $(f)_+:=\max\{0,f\}$.
In the region $\Omega_+:=\{(x,t)\in Q_1:w>0\}$, we arrive at
\begin{align}
\label{3-3}
\partial_t w&=a_{ij}(x,t,D u)w_{ij}+a_{ij,m}(x,t,D u)u_{ij}w_m+D a\cdot(e+2\rho D u) b_{ij}(D u)u_{ij} \nonumber\\
&\quad-2\rho a_{ij}(x,t,D u)u_{ki}u_{kj}.
\end{align}
Observe that
\begin{align*}
&\quad a_{ij,m}(x,t,\eta)\\
&=\big[q(|\eta|^2+\varepsilon^2)^{\frac{q}{2}-1}\eta_m+sa(x,t)(|\eta|^2+\varepsilon^2)^{\frac{s}{2}-1}\eta_m\big]
\left(\delta_{ij}+(p-2)\frac{\eta_i \eta_j}{|\eta|^2+\varepsilon^2}\right)\\
&\quad+\big[(|\eta|^2+\varepsilon^2)^\frac{q}{2}+a(x,t)(|\eta|^2+\varepsilon^2)^\frac{s}{2}\big]
(p-2)\left(\frac{\delta_{im}\eta_j+\delta_{jm}\eta_i}{|\eta|^2+\varepsilon^2}-\frac{2\eta_i\eta_j\eta_m}{(|\eta|^2+\varepsilon^2)^2}\right).
\end{align*}
Due to $|D u|>\frac{l}{2}$ in $\Omega_+$ and $|D u|\leq1$, it follows that, in $\Omega_+$,
\begin{equation}
\label{3-4}
|a_{ij,m}(x,t,D u)|\leq\begin{cases}C l^{-1} &{ \text{ if  }  q\geq 0},\\[2mm]
Cl^{q-1} &{ \text{ if} -1<q<0},
\end{cases}
\end{equation}
where $C$ depends only on $p,q,s,a^+$. We then calculate
\begin{align}
\label{3-5}
&\quad D a\cdot(e+2\rho D u)b_{ij}(D u)u_{ij} \nonumber\\
&\leq(1+2\rho)H(|D u|^2+\varepsilon^2)^\frac{s}{2}\left|\mathrm{tr}(D^2u)+(p-2)(|D u|^2+\varepsilon^2)^{-1}\langle D^2uD u,D u\rangle\right| \nonumber\\
&\leq3(n+|p-2|)H\|D^2u\|(|D u|^2+\varepsilon^2)^\frac{s}{2} \nonumber\\
&\leq\begin{cases}CH\|D^2u\| &{ \text{if } s\geq 0},\\[2mm]
CH\|D^2u\|l^s &{\text{if } -1<s<0},
\end{cases}
\end{align}
where $H:=\|D a\|_{L^\infty(Q_1)}$ and $C$ depends only on $n,p,s$. We next estimate the term $2\rho a_{ij}(x,t,D u)u_{ki}u_{kj}$ as
\begin{align}
\label{3-6}
&\quad 2\rho\big[(|D u|^2+\varepsilon^2)^\frac{q}{2}+a(x,t)(|D u|^2+\varepsilon^2)^\frac{s}{2}\big]\left(\delta_{ij}+(p-2)\frac{u_i u_j}{|D u|^2+\varepsilon^2}\right)u_{ki}u_{kj} \nonumber\\
&=2\rho\big[(|D u|^2+\varepsilon^2)^\frac{q}{2}+a(x,t)(|D u|^2+\varepsilon^2)^\frac{s}{2}\big]\left(\|D^2u\|^2+(p-2)\frac{|D^2u D u|^2}{|D u|^2+\varepsilon^2}\right) \nonumber\\
&\geq2\min\{1,p-1\}\rho(|D u|^2+\varepsilon^2)^\frac{q}{2}\|D^2u\|^2 \nonumber\\
&\geq\begin{cases}C\rho l^q\|D^2u\|^2 &{ \text{if } q\geq 0},\\[2mm]
C\rho\|D^2u\|^2 &{\text{if } -1<q<0},
\end{cases}
\end{align}
where $C$ depends only on $p,q$. Hence merging these estimates \eqref{3-3}--\eqref{3-6} and using Cauchy-Schwarz inequality, for $q\geq0$, in $\Omega_+$ we derive
\begin{align*}
\partial_tw&\leq a_{ij}(x,t,D u)w_{ij}+Cl^{-1}|D w|\sum^n_{i,j}|u_{ij}|+CH\|D^2u\|-Cl^{q+1}\|D^2u\|^2\\
&\leq a_{ij}(x,t,D u)w_{ij}+\epsilon\|D^2u\|^2+\frac{C^2}{\epsilon l^2}|D w|^2+\epsilon\|D^2u\|^2+\frac{C^2H^2}{\epsilon}-Cl^{q+1}\|D^2u\|^2\\
&\leq a_{ij}(x,t,D u)w_{ij}+Cl^{-q-3}|D w|^2+\widehat{C}l^{-q-1}H^2,
\end{align*}
by choosing $\epsilon=\frac{1}{2}Cl^{q+1}$. For $-1<q<0\leq s$, in $\Omega_+$  we have
\begin{align*}
\partial_tw&\leq a_{ij}(x,t,D u)w_{ij}+Cl^{q-1}|D w|\sum^n_{i,j}|u_{ij}|+CH\|D^2u\|-C\rho\|D^2u\|^2\\
&\leq a_{ij}(x,t,D u)w_{ij}+\epsilon\|D^2u\|^2+\frac{C^2}{\epsilon l^{2(1-q)}}|D w|^2+\epsilon\|D^2u\|^2+\frac{C^2H^2}{\epsilon}-C\rho\|D^2u\|^2\\
&\leq a_{ij}(x,t,D u)w_{ij}+Cl^{2q-3}|D w|^2+\widehat{C}l^{-1}H^2,
\end{align*}
by letting $\epsilon=\frac{1}{2}C\rho$. Finally, for $-1<s<0$, in $\Omega_+$ we get
\begin{align*}
\partial_tw&\leq a_{ij}(x,t,D u)w_{ij}+Cl^{q-1}|D w|\sum^n_{i,j}|u_{ij}|+CH\|D^2u\|l^s-C\rho\|D^2u\|^2\\
&\leq a_{ij}(x,t,D u)w_{ij}+\epsilon\|D^2u\|^2+\frac{C^2}{\epsilon l^{2(1-q)}}|D w|^2+\epsilon\|D^2u\|^2+\frac{C^2H^2}{\epsilon l^{-2s}}-C\rho\|D^2u\|^2\\
&\leq a_{ij}(x,t,D u)w_{ij}+Cl^{2q-3}|D w|^2+\widehat{C}l^{2s-1}H^2,
\end{align*}
by selecting $\epsilon=\frac{1}{2}C\rho$ again. In the previous formulations, the constants $C$ and $\widehat{C}$ separately depend on $n,p,q,s,a^+$ and $n,p,q,s$. Therefore, we can see that $w$ satisfies in the viscosity sense that
\begin{equation}
\label{3-7}
\partial_tw\leq\tilde{a}_{ij}w_{ij}+C_1(l)|D w|^2+C_2(l)H^2,
\end{equation}
where
\begin{equation*}
\tilde{a}_{ij}(x,t)=\begin{cases}a_{ij}(x,t,D u(x,t)) &{ \text{if } (x,t)\in\Omega_+},\\[2mm]
\delta_{ij} &{ \text{elsewhere}}
\end{cases}
\end{equation*}
and
\begin{equation*}
C_1(l)=\begin{cases}c_1l^{-q-3} &{ \text{if } q\geq 0},\\[2mm]
c_1l^{2q-3} &{ \text{if } -1<q<0},
\end{cases}
\end{equation*}

\begin{equation*}
C_2(l)=\begin{cases}c_2l^{-q-1} &{ \text{if } q\geq 0},\\[2mm]
c_2l^{-1} &{ \text{if } -1<q<0\leq s},\\[2mm]
c_2l^{2s-1} &{ \text{if } -1<s<0}
\end{cases}
\end{equation*}
with $c_1$ depending on $n,p,q,s,a^+$ and $c_2$ depending only on $n,p,q,s$. Here we notice that, since $l\in(\frac{1}{2},1)$, the coefficient $\tilde{a}_{ij}(x,t)$ is uniformly parabolic, i.e., there exist two constants $0<\lambda\leq\Lambda<\infty$ such that $\lambda I\leq \tilde{a}_{ij}(x,t)\leq \Lambda I$ for all $(x,t)\in Q_1$. Indeed, we find that $\lambda$ depends only on $p,q$ and $\Lambda$ depends on $p,q,s$ and $a^+$. For simplicity, \eqref{3-7} is reformulated as
$$
\partial_tw\leq\tilde{a}_{ij}w_{ij}+C_1(l)|D w|^2+\overline{c},
$$
with $\overline{c}\equiv C_2(l)H^2$. Set
$$
W=1-l+\rho+\overline{c},
$$
and
$$
U=\frac{1}{\nu}\left(1-e^{\nu(w-\overline{c}t-W)}\right) \quad\text{with }  \nu>0.
$$
We could determine $\nu>0$, which depends on $n,p,q,s, a^+$ and $l$, such that
$$
\partial_tU\geq\tilde{a}_{ij}U_{ij} \quad\text{in } Q_1
$$
in the viscosity sense. Obviously, $U\geq0$ in $Q_1$.

If $D u\cdot e\leq l$, then it follows from the assumption in the statement that
$$
|\{(x,t)\in Q_1:U\geq \nu^{-1}(1-e^{\nu(l-1)})\}|>\mu|Q_1|.
$$
Thereby, we can infer from Proposition 2.3 in \cite{JS17} that there are two constants $\tau_0, \gamma_0>0$ such that
$$
U\geq \nu^{-1}(1-e^{\nu(l-1)})\gamma_0 \quad\text{in } Q_{\tau_0},
$$
where $\tau_0$ depends only on $n,\mu$ and $\gamma_0$ depends on $n,\mu,p,q,s$ and $a^+$. Moreover, since $w-\overline{c}t\leq W$, we could readily get
$$
U\leq W-w+\overline{c}t.
$$
Thus, in $Q_{\tau_0}$ we have
$$
D u\cdot e+\rho|D u|^2\leq1+\rho-\nu^{-1}(1-e^{\nu(l-1)})\gamma_0+\overline{c}+\overline{c}t,
$$
which leads to
$$
D u\cdot e+\rho(D u\cdot e)^2\leq1+\rho-\nu^{-1}(1-e^{\nu(l-1)})\gamma_0+\overline{c}.
$$
It yields that
$$
D u\cdot e\leq\frac{-1+\sqrt{1+4\rho(1+\rho-\nu^{-1}(1-e^{\nu(l-1)})\gamma_0+\overline{c})}}{2\rho} \quad\text{in }  Q_{\tau_0}.
$$
If
$$
H<(C_2^{-1}(l)\nu^{-1}(1-e^{\nu(l-1)})\gamma_0)^\frac{1}{2}=:\kappa (\leq1),
$$
which implies that $\overline{c}<\nu^{-1}(1-e^{\nu(l-1)})\gamma_0$, then there holds that
$$
D u\cdot e\leq 1-\delta \quad\text{in } Q_{\tau_0},
$$
where $\delta>0$ depends on $n,p,q,s,a^+,\mu$ and $l$. Briefly, when the upper bound on $|D a(x,t)|$ is small enough depending on $n,p,q,s,a^+,\mu$ and $l$, we arrive at $D u\cdot e\leq 1-\delta$ in $Q_{\tau_0}$. Finally, we select
\begin{equation*}
\tau=\begin{cases}\tau_0(1-\delta)^\frac{q}{2} &{ \text{if } q\geq 0},\\[2mm]
\tau_0 &{ \text{if } -1<q<0}
\end{cases}
\end{equation*}
such that $Q_\tau^{1-\delta}\subset Q_{\tau_0}$. We now complete the proof.
\end{proof}

\begin{remark}
Observe that the selection of $\tau$ and $\delta$ above implies that
$$
\tau<(1-\delta)^\frac{q}{2} \quad \text{when } q\geq 0.
$$
In the rest of this work, we shall choose such smaller $\tau$ that
\begin{equation}
\label{3-8}
\tau<(1-\delta)^{1+q} \quad \text{for any } q\geq -1.
\end{equation}
\end{remark}

If Lemma \ref{lem3-1} holds true in all directions $e \in \mathbb{S}^{n-1}$, then it in effect indicates a reduction in the oscillation of $D u$ in a smaller parabolic cylinder. This content is stated by the forthcoming corollary.

\begin{corollary}
\label{cor3-2}
Let the conditions \eqref{1-1} and \eqref{1-2} be in force. Suppose that $u$ is a smooth solution of \eqref{2-1} with $\varepsilon\in(0,1)$ such that $|D u|\leq1$ in $Q_1$. For each $l\in(0,1)$ and $\mu>0$, if $0\leq a(x,t)\in C^1(Q_1)$ and $\|D a\|_{L^\infty(Q_1)}\leq \kappa$, where $\kappa\in(0,1]$ is a small enough quantity depending on $n,p,q,s,a^+,\mu$ and $l$, then there exist $\tau\in(0,\frac{1}{4})$  and $\delta>0$ that both depend upon $n,p,q,s,a^+,\mu$ and $l$ such that for all nonnegative integer $k\leq \log\varepsilon/\log(1-\delta)$ if
\begin{equation}
\label{3-9}
\left|\left\{(x,t)\in Q_{\tau^i}^{(1-\delta)^i}: D u\cdot e\leq l(1-\delta)^i\right\}\right|>\mu\left|Q_{\tau^i}^{(1-\delta)^i}\right|
\end{equation}
for all $e\in \mathbb{S}^{n-1}$ and $i=0,1,2,\cdots,k$, then one has
$$
|D u|<(1-\delta)^{i+1}
$$
in $Q_{\tau^{i+1}}^{(1-\delta)^{i+1}}$ for $i=0,1,2,\cdots,k$.
\end{corollary}

\begin{remark}
Notice that we could further impose on $\delta$ that $\delta<\frac{1}{2}$ and $\delta<1-\tau$.
\end{remark}

\begin{proof}
Argue by induction. If $i=0$, it follows from Lemma \ref{lem3-1} that $D u\cdot e<1-\delta$ in $Q_\tau^{1-\delta}$ for all $e\in \mathbb{S}^{n-1}$, which leads to $|D u|<1-\delta$ in $Q_\tau^{1-\delta}$. We now suppose that this claim holds true for $i=0,1,2,\cdots,k-1$. Next, we shall verify it for $i=k$. Define
$$
w(x,t)=\frac{1}{\tau^k(1-\delta)^k}u(\tau^kx,\tau^{2k}(1-\delta)^{-kq}t),  \quad  (x,t)\in Q_1.
$$
We can readily check that $w$ solves in the viscosity sense
\begin{align*}
\partial_tw=\big[(|D w|^2+\hat{\varepsilon}^2)^\frac{q}{2}+\hat{a}(x,t)(|D w|^2+\hat{\varepsilon}^2)^\frac{s}{2}\big]
\left(\delta_{ij}+(p-2)\frac{w_i w_j}{|D w|^2+\hat{\varepsilon}^2}\right)w_{ij}
\end{align*}
in $Q_1$, where
$$
\hat{a}(x,t)=(1-\delta)^{k(s-q)}a(\tau^kx,\tau^{2k}(1-\delta)^{-qk}t) \quad \text{and} \quad\hat\varepsilon^2=\frac{\varepsilon^2}{(1-\delta)^{2k}}.
$$
Moreover, there holds that
$$
|D w|\leq1 \quad\text{in } Q_1
$$
and
$$
|\{(x,t)\in Q_1:D w\cdot e\leq l\}|>\mu|Q_1| \quad\text{for all } e\in \mathbb{S}^{n-1},
$$
by the induction assumption. Additionally, note that $\varepsilon\leq(1-\delta)^k$. Applying Lemma \ref{lem3-1}, we obtain
$$
D w\cdot e\leq1-\delta \quad\text{in } Q_\tau^{1-\delta} \quad\text{for all } e\in \mathbb{S}^{n-1},
$$
which implies that $|D w|\leq 1-\delta$ in $Q_\tau^{1-\delta}$. Rescaling back, it yields that
$$
|D u|<(1-\delta)^{k+1} \quad\text{in }  Q_{\tau^{k+1}}^{(1-\delta)^{k+1}}.
$$
We conclude the proof.
\end{proof}

\begin{remark}
\label{rem3-2-1}
In order to derive the reduction of oscillation of $|D u|$, we ask that $\|D a\|_{L^\infty(Q_1)}$ is smaller than $\kappa$ ($\kappa\leq1$ sufficiently small) in Lemmas \ref{lem3-1} and \ref{cor3-2}. In fact, we can suppose initially that $\|D_{x,t} a\|_{L^\infty(Q_1)}\leq \kappa$. These two lemmas still hold, when $\|D a\|_{L^\infty(Q_1)}\leq\kappa$ is substituted by $\|D_{x,t}a\|_{L^\infty(Q_1)}\leq \kappa$ in Lemmas \ref{lem3-1} and \ref{cor3-2}.
\end{remark}

If the iteration above can be carried out infinitely, then we will readily conclude the H\"{o}lder continuity of $D u$ at the origin $(0,0)$. Nevertheless, unless $D u(0,0)=0$, the iteration shall stop unavoidably at some step, that is, for some nonnegative integer $k$ the condition \eqref{3-9} is not true in some direction $e\in \mathbb{S}^{n-1}$. In this case, we will infer that $u$ can be approximated by a linear function, then making use of a conclusion on regularity of small perturbation solutions from \cite{Wang13} to prove the H\"{o}lder regularity for $D u$. Now we first study how the solution is close to a linear function.

\begin{lemma}
\label{lem3-2}
Let the conditions \eqref{1-1} and \eqref{1-2} be in force. Suppose that $u\in C(\overline{Q_1})$ is a smooth solution to \eqref{2-1} with $\varepsilon\in(0,1)$ such that $|D u|\leq M$ in $Q_1$. Let $0\leq a(x,t)\leq a^+$. If for any $t\in[-1,0]$ it holds that
$$
\mathrm{osc}_{B_1}u(\cdot,t)\leq L
$$
with $L$ being a positive constant, then
\begin{equation*}
\mathrm{osc}_{Q_1}u(x,t)\leq \begin{cases} CL &{ \text{if }  q\geq 0},\\[2mm]
C(L+L^{1+q}+L^{1+s}) &{ \text{if } -1<q<0},
\end{cases}
\end{equation*}
where $C>0$ depends upon $n,p,q,s,a^+$ and $M$.
\end{lemma}

\begin{proof}
This proof is similar to that of Lemma 4.4 in \cite{IJS19}. We only give the sketch of proof here.
If $q\geq 0$, for $a_{ij}$ defined in \eqref{3-2} we can find that
$$
|a_{ij}|\leq \overline{\Lambda}:=\big[(1+M^2)^\frac{q}{2}+a^+(1+M^2)^\frac{s}{2}\big]\max\{1,p-1\}.
$$
So this claim can be concluded by the same proof of Lemma 4.3 in \cite{JS17}.

When $-1<q<0$, we define the comparison functions as follows
$$
\overline{w}(x,t)=\overline{b}+\Lambda L^{1+q}t+2L|x|^\beta,
$$
$$
\underline{w}(x,t)=\underline{b}-\Lambda L^{1+q}t-2L|x|^\beta,
$$
where $\beta=\frac{q+2}{q+1}$ and $\Lambda$ will be determined later.  Here $\overline{b}$ is chosen so that $\overline{w}(\cdot,-1)\geq u(\cdot,-1)$ in $B_1$ and $\overline{w}(\overline{x},-1)= u(\overline{x},-1)$ at some point $\overline{x}\in \overline{B_1}$. Correspondingly, $\underline{b}$ is chosen so that $\underline{w}(\cdot,-1)\leq u(\cdot,-1)$ in $B_1$ and $\underline{w}(\underline{x},-1)= u(\underline{x},-1)$ at some point $\underline{x}\in \overline{B_1}$. Then $\overline{b}-\underline{b}\leq L+2\Lambda L^{1+q}$.
By direct computations,
\begin{align*}
&\quad\big[(|D \overline{w}|^2+\varepsilon^2)^\frac{q}{2}+a(x,t)(|D \overline{w}|^2+\varepsilon^2)^\frac{s}{2}\big]\left(\delta_{ij}+(p-2)\frac{\overline{w}_i \overline{w}_j}{|D \overline{w}|^2+\varepsilon^2}\right)\overline{w}_{ij}\\
&\leq\big[((2L\beta|x|^{\beta-1})^2+\varepsilon^2)^\frac{q}{2}+a(x,t)((2L\beta|x|^{\beta-1})^2
+\varepsilon^2)^\frac{s}{2}\big]2pn\beta(\beta-1)L|x|^{\beta-2}\\
&\leq(2\beta)^{q+1}(\beta-1)pn\big[1+a^+(1+(2L\beta)^2)^\frac{s-q}{2}\big]L^{q+1}.
\end{align*}
Hence, if we choose
$$
\Lambda=(2\beta)^{q+1}(\beta-1)pn\big[1+a^+(1+(2L\beta)^2)^\frac{s-q}{2}\big]+1,
$$
then $\overline{w}$ is a strict supersolution to \eqref{2-1}. Analogously, $\underline{w}$ is a strict subsolution. Next, we can prove that $\overline{w}\geq u\geq\underline{w}$ in $Q_1$. For the details, one can see \cite[Lemma 4.4]{IJS19}. Finally we have
$$
\mathrm{osc}_{Q_1}u(x,t)\leq \overline{b}-\underline{b}+4L\leq 2\Lambda L^{q+1}+5L=:C(L+L^{1+q}+L^{1+s}).
$$
The proof now is completed.
\end{proof}

\begin{lemma}
\label{lem3-3}
Let $e\in \mathbb{S}^{n-1}$ and $0<\sigma<\frac{1}{8}$. Let the conditions \eqref{1-1} and \eqref{1-2} be in force. Suppose that $u\in C(\overline{Q_1})$ is a smooth solution to \eqref{2-1} with $\varepsilon\in(0,1)$ in $Q_1$. Let $0\leq a(x,t)\leq a^+$. If for any $t\in[-1,0]$ it holds that
$$
\mathrm{osc}_{x\in B_1}(u(x,t)-x\cdot e)\leq \sigma,
$$
then one has
$$
\mathrm{osc}_{(x,t)\in Q_1}(u(x,t)-x\cdot e)\leq C\sigma,
$$
where $C>0$ depends only on $n,p,q,s$ and $a^+$.
\end{lemma}

\begin{proof}
Denote
$$
\overline{w}(x,t)=\overline{b}+x\cdot e+\Lambda \sigma t+2\sigma|x|^2,
$$
$$
\underline{w}(x,t)=\underline{b}+x\cdot e-\Lambda \sigma t-2\sigma|x|^2,
$$
with $\Lambda>0$ to be fixed later. Here, the choices of $\overline{b},\underline{b}$ are the same as that in the proof of Lemma \ref{lem3-2}. Then we get $\overline{b}-\underline{b}\leq (2\Lambda+1)\sigma$. Owing to $\sigma<\frac{1}{8}$, by simple calculation, we obtain
$$
\frac{1}{2}\leq |D \overline{w}(x,t)|, |D \underline{w}(x,t)|\leq \frac{3}{2}, \quad \text{for } (x,t)\in \overline{Q_1}.
$$
Therefore, it follows that
$$
a_{ij}(x,t,D \overline{w}(x,t))\leq A_0I
$$
and
$$
a_{ij}(x,t,D \underline{w}(x,t))\leq A_0I,
$$
where the notation $a_{ij}$ is from \eqref{3-2}, and the constant $A_0$ depends on $p,q,s,a^+$.

We next pick $\Lambda=5nA_0$ and then show that
$$
\underline{w}\leq u\leq \overline{w}
$$
in $Q_1$. For the details, we refer to \cite[Lemma 4.5]{IJS19}. Finally, we arrive at
$$
\mathrm{osc}_{(x,t)\in Q_1}(u(x,t)-x\cdot e)\leq \sup_{Q_1}(\overline{w}-x\cdot e)-\inf_{Q_1}(\underline{w}-x\cdot e)\leq \overline{b}-\underline{b}+4\sigma \leq (2\Lambda+5)\sigma.
$$
We now finish the proof.
\end{proof}

Now putting together Lemmas \ref{lem3-2} and \ref{lem3-3} with Lemma \ref{lem2-2}, we will conclude the following result which states that if $D u$ is close to a unit vector in a large portion of $Q_1$, then $u$ is close to some linear function. Since $\|D_{x,t} a(x,t)\|_{L^\infty(Q_1)}$ is assumed to be smaller than $\kappa\leq1$ in Lemmas \ref{lem3-1} and \ref{cor3-2} (see Remark \ref{rem3-2-1}), in the lemma below we will suppose $\|D_{x,t} a(x,t)\|_{L^\infty(Q_1)}\leq 1$ is in force for simplicity.

\begin{lemma}
\label{lem3-4}
Let the assumptions \eqref{1-1}--\eqref{1-3} be in force and let $\omega$ be a positive number. Suppose that $|D_{x,t} a(x,t)|\leq 1$ in $Q_1$. Assume that $u$ is a smooth solution of \eqref{2-1} with $0<\varepsilon<1$ satisfying $|D u|\leq 1$ in $Q_1$, and that for some $e\in \mathbb{S}^{n-1}$ and two quantities $\varepsilon_0,\varepsilon_1>0$, there holds that
$$
|\{(x,t)\in Q_1:|D u-e|>\varepsilon_0\}|\leq \varepsilon_1.
$$
Then if $\varepsilon_0,\varepsilon_1$ are small enough, there is a real number $d\in\mathbb{R}$ such that
$$
|u(x,t)-d-e\cdot x|\leq \omega
$$
for $(x,t)\in Q_{1/2}$. Here both $\varepsilon_0$ and $\varepsilon_1$ depend upon $n,p,q,s,a^-,a^+$ and $\omega$.
\end{lemma}

\begin{proof}
Set
$$
f(t):=|\{x\in B_1:|D u(x,t)-e|>\varepsilon_0\}|,
$$
$$
D:=\{t\in(-1,0):f(t)\geq \sqrt{\varepsilon_1}\}.
$$
We can easily get
$$
\int^0_{-1}f(t)\,dt\leq \varepsilon_1 \quad\text{and} \quad |D|\leq \sqrt{\varepsilon_1}.
$$
Thereby, it yields that
\begin{equation}
\label{3-10}
|\{x\in B_1:|D u(x,t)-e|>\varepsilon_0\}|\leq \sqrt{\varepsilon_1},
\end{equation}
for any $t\in (-1,0]\setminus D$ with $|D|\leq \sqrt{\varepsilon_1}$. Applying \eqref{3-10} and Morrey's inequality to get for all $t\in (-1,0]\setminus D$,
\begin{align}
\label{3-11}
&\quad \mathrm{osc}_{B_{1/2}}(u(x,t)-e\cdot x) \nonumber\\
&\leq C\|D u-e\|_{L^{2n}(B_1)} \nonumber\\
&=C\left(\int_{\{x\in B_1:|D u(x,t)-e|\leq\varepsilon_0\}}+\int_{\{x\in B_1:|D u(x,t)-e|>\varepsilon_0\}}|D u-e|^{2n}\,dx\right)^\frac{1}{2n} \nonumber\\
&\leq C(\varepsilon_0+\varepsilon_1^\frac{1}{4n}),
\end{align}
where $C$ depends only on $n$.

On the other hand, we can see that
$$
\mathrm{osc}_{B_1}u(\cdot,t)\leq 2
$$
for $t\in (-1,0]$, due to $|D u|\leq 1$ in $Q_1$. Then we employ Lemma \ref{lem3-2} to deduce
$$
\mathrm{osc}_{Q_1}u(x,t)\leq C
$$
with $C$ depending on $n,p,q,s$ and $a^+$. Observe that $u(x,t)-u(0,0)$ solves \eqref{2-1} as well, and obviously
$$
\|u(x,t)-u(0,0)\|_{L^\infty(Q_1)}\leq C.
$$
Via applying Lemma \ref{lem2-2} to $u(x,t)-u(0,0)$, it follows that
$$
 \sup_{\stackrel{(x,t),(x,s)\in Q_1}{t \neq s}}\frac{|u(x,t)-u(x,s)|}{|t-s|^\gamma}\leq C,
$$
where
\begin{equation*}
\gamma=\begin{cases}\frac{1}{\beta(1+s)-s} &{ \text{if } -1<q<0},\\[2mm]
\frac{1}{2} &{ \text{if } 0\leq q}
\end{cases}
\end{equation*}
with $\beta=\frac{q+2}{q+1}$, and the constant $C$ depends only on $n,p,q,s,a^-$ and $a^+$. Thus by $|D|\leq \sqrt{\varepsilon_1}$ and \eqref{3-11}, we derive
$$
\mathrm{osc}_{x\in B_\frac{1}{2}}(u(x,t)-e\cdot x)\leq C\big(\varepsilon_0+\varepsilon_1^\frac{1}{4n}+\varepsilon_1^\frac{\gamma}{2}\big)
$$
for any $t\in(-\frac{1}{4},0]$. It follows from Lemma \ref{lem3-3} that
$$
\mathrm{osc}_{(x,t)\in Q_\frac{1}{2}}(u(x,t)-e\cdot x)\leq C\big(\varepsilon_0+\varepsilon_1^\frac{1}{4n}+\varepsilon_1^\frac{\gamma}{2}\big),
$$
if $\varepsilon_0$ and $\varepsilon_1$ are small enough. Consequently, if $\varepsilon_0,\varepsilon_1$ are sufficiently small, then there is a constant $d\in\mathbb{R}$ satisfying
$$
|u(x,t)-d-e\cdot x|\leq \omega
$$
for any $(x,t)\in Q_{1/2}$.
\end{proof}

Taking into account that the framework of equation \eqref{main} is analogous to that of
$$
\partial_tu=|D u|^q\left(\delta_{ij}+(p-2)\frac{u_iu_j}{|D u|^2}\right)u_{ij},
$$
we may obtain, with the help of the same tool that is Corollary 1.2 in \cite{Wang13}, the regularity of small perturbation solutions to \eqref{2-1} as well.

\begin{proposition}
\label{pro3-5}
Let $\beta\in(0,1)$. Suppose that $u$ is a smooth solution to \eqref{2-1} in $Q_1$. Then there are $\omega>0$ (small) and $C>0$ (large), both depending on $n,p,q,s,a^+,\|D_{x,t}a(x,t)\|_{L^\infty(Q_1)}$ and $\beta$, such that if a linear function $L(x)$ with $\frac{1}{2}\leq |D L|\leq2$ fulfills
$$
\|u(x,t)-L(x)\|_{L^\infty(Q_1)}\leq \omega,
$$
then
$$
\|u-L\|_{C^{2,\beta}(Q_{1/2})}\leq C.
$$
\end{proposition}

\begin{proof}
Note that $L(x)$  is also a solution to \eqref{2-1}. So we could infer this conclusion by Corollary 1.2 in \cite{Wang13}.
\end{proof}

Once we show that $u$ is close to some linear function, then the H\"{o}lder regularity for $D u$ follows from the above proposition. From Remark \ref{rem3-2-1}, we have known that $\|D_{x,t}a(x,t)\|_{L^\infty(Q_1)}$ is less than $\kappa$ ($\kappa\leq1$), so in Proposition \ref{pro3-5} we may replace $\|D_{x,t}a(x,t)\|_{L^\infty(Q_1)}$ by $1$ so that $\omega, C$ do not depend on $\|D_{x,t}a(x,t)\|_{L^\infty(Q_1)}$.

In the sequel, we shall give a uniformly {\em a priori} H\"{o}lder estimate on the solution of \eqref{2-1}.

\begin{theorem}[{\em A priori} H\"{o}lder gradient estimate]
\label{thm3-6}
Let the assumptions \eqref{1-1}--\eqref{1-3} be in force and let $\varepsilon\in(0,1)$. Suppose that $\|D_{x,t} a(x,t)\|_{L^\infty(Q_1)}\leq \kappa$, where $\kappa\in(0,1]$ is a small constant depending on $n,p,q,s,a^-,a^+$. Assume that $u$ is a smooth solution of \eqref{2-1} satisfying $|D u|\leq 1$ in $Q_1$. Then there are two constants $\alpha,C>0$, both of which depend on $n,p,q,s,a^-$ and $a^+$, such that the following estimates hold:
$$
|D u(x,t)-D u(y,s)|\leq C\big(|x-y|^\alpha+|t-s|^{\frac{\alpha}{2-\alpha q}}\big)
$$
and
$$
|u(x,t)-u(x,s)|\leq C|t-s|^\frac{1+\alpha}{2-\alpha q}
$$
for any $(x,t),(y,s),(x,s)\in Q_{\frac{1}{2}}$.
\end{theorem}

\begin{proof}
This proof is similar to that of Theorem 4.8 in \cite{IJS19}. However, for the sake of completeness and convenience, we give the details of proof here. We first demonstrate the H\"{o}lder continuity of $D u$ at $(0,0)$ and the H\"{o}lder continuity of $u$ in $t$ at (0,0). Then, by standard translation arguments, the interior H\"{o}lder regularity follows.

We choose $\omega$ as the one in Proposition \ref{pro3-5} with $\|D_{x,t} a(x,t)\|_{L^\infty(Q_1)}$ replaced by $1$ and $\beta=\frac{1}{2}$. And then for this $\omega$ we pick two quantities $\varepsilon_0,\varepsilon_1>0$  small so that Lemma \ref{lem3-4} holds true. Now we fix
$$
l=1-\frac{\varepsilon_0^2}{2} \quad \text{and} \quad \mu=\frac{\varepsilon_1}{|Q_1|}.
$$
Here we observe that if for arbitrary $e\in \mathbb{S}^{n-1}$
$$
|\{(x,t)\in Q_1:D u\cdot e\leq l\}|\leq\mu|Q_1|,
$$
then
$$
|\{(x,t)\in Q_1:|D u- e|>\varepsilon_0\}|\leq\varepsilon_1.
$$
This estimate will be used later.

Let $\tau$ and $\delta$ be two positive constants coming from Corollary \ref{cor3-2}. By $[\log \varepsilon/\log(1-\delta)]$, we mean the integer part of $\log\varepsilon/\log(1-\delta)$. Let $k$ be either $[\log\varepsilon/\log(1-\delta)]$ or the minimum nonnegative integer that makes condition \eqref{3-9} false, whichever is smaller. Then it follows from Corollary \ref{cor3-2} that
$$
|D u(x,t)|\leq (1-\delta)^m \quad\text{in } Q_{\tau^m}^{(1-\delta)^m}
$$
for $m=0,1,\cdots,k$. When
$$
(x,t)\in Q_{\tau^m}^{(1-\delta)^m}\setminus Q_{\tau^{m+1}}^{(1-\delta)^{m+1}},
$$
we know that
$$
|x|\geq \tau^{m+1} \quad\text{or} \quad  |t|\geq (1-\delta)^{-(m+1)q}\tau^{2(m+1)}.
$$
Thus, by taking $\alpha=\frac{\log(1-\delta)}{\log \tau}$, it yields that
$$
|x|^\alpha\geq (1-\delta)^{m+1} \quad\text{or} \quad  |t|^\frac{\alpha}{2-\alpha q}\geq (1-\delta)^{(m+1)}.
$$
Then
$$
|D u(x,t)|\leq (1-\delta)^m\leq C\left(|x|^\alpha+|t|^\frac{\alpha}{2-\alpha q}\right)
$$
with $C=\frac{1}{1-\delta}$ in $Q_{\tau^m}^{(1-\delta)^m}\setminus Q_{\tau^{m+1}}^{(1-\delta)^{m+1}}$. Hence, for each $\xi\in \mathbb{R}^n$ with $|\xi|\leq (1-\delta)^k$,
\begin{equation}
\label{3-12}
|D u(x,t)-\xi|\leq (1-\delta)^m+(1-\delta)^k\leq C\left(|x|^\alpha+|t|^\frac{\alpha}{2-\alpha q}\right)
\end{equation}
in $Q_1\setminus Q_{\tau^{k+1}}^{(1-\delta)^{k+1}}$. Observe that if $q\geq0$, \eqref{3-8} implies that $2-\alpha q>0$ and $\frac{\alpha}{2-\alpha q}<\frac{1}{2}$. For $m=0,1,\cdots,k$, set
\begin{equation}
\label{3-13}
u_m(x,t)=\frac{1}{\tau^m(1-\delta)^m} u(\tau^mx,\tau^{2m}(1-\delta)^{-mq}t), \quad (x,t)\in Q_1.
\end{equation}
We can check that $|D u_m(x,t)|\leq1$ in $Q_1$ and $u_m$ solves
\begin{align}
\label{3-14}
\partial_tu_m=\big[(|D u_m|^2+\hat{\varepsilon}^2)^\frac{q}{2}+\hat{a}(x,t)(|D u_m|^2+\hat{\varepsilon}^2)^\frac{s}{2}\big]
\left(\delta_{ij}+(p-2)\frac{\partial_i u_m \partial_ju_m}{|D u_m|^2+\hat{\varepsilon}^2}\right)\partial_{ij}u_m,
\end{align}
where
$$
\hat{a}(x,t)=(1-\delta)^{m(s-q)}a(\tau^mx,\tau^{2m}(1-\delta)^{-mq}t) \quad \text{and} \quad\hat\varepsilon=\frac{\varepsilon}{(1-\delta)^m}.
$$
Clearly, $\hat{\varepsilon}^2\leq \varepsilon^2(1-\delta)^{-2k}\leq1$. Noting the framework of this equation, we have, for any $t\in[-1,0]$,
$$
\mathrm{osc}_{B_1}u_m(\cdot,t)\leq2
$$
and further get
$$
\mathrm{osc}_{Q_1}u_m\leq C
$$
by Lemma \ref{lem3-2}. This indicates
\begin{equation}
\label{3-15}
\mathrm{osc}_{Q_{\tau^m}^{(1-\delta)^m}}u\leq C\tau^m(1-\delta)^m.
\end{equation}
In the rest of proof, we let $w(x,t)=u_k(x,t)$.

\smallskip

\textbf{Case 1.} $k=[\log\varepsilon/\log(1-\delta)]$. Then we get $(1-\delta)^{k+1}<\varepsilon\leq (1-\delta)^k$, and hence $\frac{1}{2}<1-\delta<\varepsilon(1-\delta)^{-k}\leq1$. From this, we find that, when $m=k$, \eqref{3-14} is a uniformly quasilinear parabolic equation with bounded and smooth coefficients. In view of the standard parabolic quasilinear equation theory (see, e.g. \cite[Theorem 4.4]{LSU68}) together with Schauder estimates, we may find $\zeta\in\mathbb{R}^n, |\zeta|\leq1$ such that in $Q_\tau^{1-\delta}\subset Q_{1/4}$
$$
|D w(x,t)-\zeta|\leq C(|x|+|t|^\frac{1}{2})\leq C\big(|x|^\alpha+|t|^\frac{\alpha}{2-\alpha q}\big)
$$
and
$$
|\partial_t w|\leq C,
$$
where $C>0$ depends on $n,p,q,s$ and $a^+$. Here we have utilized the fact that $\frac{\alpha}{2-\alpha q}\leq\frac{1}{2}$. Rescaling back, we arrive at
$$
|D u(\tau^kx,\tau^{2k}(1-\delta)^{-kq}t)-(1-\delta)^k\zeta|\leq C(1-\delta)^k\big(|x|^\alpha+|t|^\frac{\alpha}{2-\alpha q}\big)
$$
for $(x,t)\in Q_\tau^{1-\delta}$, which leads to
\begin{align}
\label{3-16}
|D u(x,t)-(1-\delta)^k\zeta|&\leq C(1-\delta)^k\big(\tau^{-k\alpha}|x|^\alpha+(\tau^{-2k}(1-\delta)^{kq})^\frac{\alpha}{2-\alpha q}|t|^\frac{\alpha}{2-\alpha q}\big)  \nonumber\\
&=C\big(|x|^\alpha+|t|^\frac{\alpha}{2-\alpha q}\big)
\end{align}
by $(1-\delta)\tau^{-\alpha}=1$, where $(x,t)\in Q_{\tau^{k+1}}^{(1-\delta)^{k+1}}$. Similarly,
\begin{equation}
\label{3-17}
|u(x,t)-u(x,0)|\leq C\tau^{-k}(1-\delta)^{k(1+q)}|t|
\end{equation}
with $(x,t)\in Q_{\tau^{k+1}}^{(1-\delta)^{k+1}}$. Therefore, it yields by \eqref{3-12} and \eqref{3-16} that for some vector $\xi_0\in \mathbb{R}^n$,
$$
|D u(x,t)-\xi_0|\leq C\big(|x|^\alpha+|t|^\frac{\alpha}{2-\alpha q}\big)
$$
in $Q_{1/2}$, where $C>0$ depends on $n,p,q,s,a^-$ and $a^+$. On the other hand, from \eqref{3-17} we obtain, for $|t|\leq \tau^{2l}(1-\delta)^{-lq}$ with $l\geq k+1$, that
\begin{equation}
\label{3-18}
|u(0,t)-u(0,0)|\leq C\tau^l(1-\delta)^l,
\end{equation}
where we exploited the fact that $\tau<(1-\delta)^{1+q}$ in \eqref{3-8}. By means of \eqref{3-15} and \eqref{3-18}, it follows that
$$
|u(0,t)-u(0,0)|\leq C|t|^\beta
$$
for every $t\in(-\frac{1}{4},0]$, where $\beta=\frac{1+\alpha}{2-\alpha q}$. It is easy to see that $\beta>\frac{1}{2}$ if $q>-2$. Indeed, if $t\in(-\frac{1}{4},0)$, there always exists $\tilde{l}\in \{0,1,\cdots,k,k+1,\cdots,l,\cdots\}$ such that
$$
t\in\big(-\tau^{2\tilde{l}}(1-\delta)^{-q\tilde{l}},-\tau^{2(\tilde{l}+1)}(1-\delta)^{-q(\tilde{l}+1)}\big].
$$
Via \eqref{3-15} and \eqref{3-18}, when we select $\beta$ verifying
$$
\tau(1-\delta)=(\tau^2(1-\delta)^{-q})^\beta,
$$
then we get
$$
|u(0,t)-u(0,0)|\leq C\tau^{\tilde{l}}(1-\delta)^{\tilde{l}}=\frac{C}{(\tau^2(1-\delta)^{-q})^\beta}(\tau^2(1-\delta)^{-q})^{\beta(\tilde{l}+1)}
\leq \frac{C}{\tau(1-\delta)}|t|^\beta.
$$

\smallskip

\textbf{Case 2.} $k<[\log\varepsilon/\log(1-\delta)]$. Then for some $e\in \mathbb{R}^n$,
\begin{equation}
\label{3-19}
\left|\left\{(x,t)\in Q_{\tau^k}^{(1-\delta)^k}: D u\cdot e\leq l(1-\delta)^k\right\}\right|\leq\mu\left|Q_{\tau^k}^{(1-\delta)^k}\right|.
\end{equation}
Additionally,
$$
|D u|<(1-\delta)^l \quad\text{in } Q_{\tau^l}^{(1-\delta)^l} \quad\text{for all } l=0,1,\cdots,k.
$$
We can easily see that $w$ ($w=u_k$) satisfies $|D w|\leq 1$ and solves equation \eqref{3-14} in $Q_1$. By virtue of \eqref{3-19} and the selections of $l$ and $\mu$, we get
$$
|\{(x,t)\in Q_1:|D u- e|>\varepsilon_0\}|\leq\varepsilon_1.
$$
It thus follows from Lemma \ref{lem3-4} that there is $d\in\mathbb{R}$ satisfying
$$
|w(x,t)-d-e\cdot x|\leq \omega \quad \text{for all } (x,t)\in Q_{1/2}.
$$
By Proposition \ref{pro3-5}, there is $b\in\mathbb{R}^n$ fulfilling
$$
|D w-b|\leq C(|x|+\sqrt{|t|})
$$
and
$$
|\partial_t w|\leq C
$$
in $Q_\tau^{1-\delta}\subset Q_{1/4}$. As in Case $1$, we also arrive at
$$
|D u(x,t)-\xi_1|\leq C\big(|x|^\alpha+|t|^\frac{\alpha}{2-\alpha q}\big) \quad\text{in } Q_{1/2}
$$
for some vector $\xi_1\in\mathbb{R}^n$ with $|\xi_1|\leq1$, and
$$
|u(0,t)-u(0,0)|\leq C|t|^\beta \quad\text{for } t\in \left(-\frac{1}{4},0\right],
$$
where $C>0$ depends on $n,p,q,s,a^-$ and $a^+$.

In summary, we have showed that there are $\alpha,C>0$, depending on $n,p,q,s,a^-$ and $a^+$, as well as $\xi\in\mathbb{R}^n$ with $|\xi|\leq1$, such that
$$
|D u(x,t)-\xi|\leq C\big(|x|^\alpha+|t|^\frac{\alpha}{2-\alpha q}\big) \quad\text{in } Q_{1/2}
$$
and
$$
|u(0,t)-u(0,0)|\leq C|t|^\beta \quad\text{for } t\in \left(-\frac{1}{4},0\right].
$$
Finally the claim follows by using the standard translation arguments.
\end{proof}

Next we are going to use the solution of equation \eqref{2-1} to approximate the solution of \eqref{main}. Before that, we need some crucial results on viscosity solutions such as the boundary estimates, apart from the known comparison principle and stability. We prescribe that the assumptions \eqref{1-1} and \eqref{1-2} hold in the following two conclusions. Here two notations are introduced, for two real numbers $a$ and $b$, $a\vee b = \max\{a, b\}$ and $a\wedge b = \min\{a, b\}$.

\begin{proposition}[Boundary estimates]
\label{pro3-7}
Assume that $u\in C(\overline{Q_1})$ is a solution to \eqref{2-1} with $\varepsilon\in(0,1)$ and $a(x,t)$ satisfying \eqref{1-3} and $\|D_{x,t}a(x,t)\|_{L^\infty(Q_1)}\leq1$, and that $\varphi:=u\mid_{\partial_pQ_1}$ possesses a modulus of continuity denoted by $\rho$. Then there is another modulus of continuity $\rho^*$, which depends on $n,p,q,s,a^-,a^+, \rho$ and $\|\varphi\|_{L^\infty(\partial_pQ_1)}$, such that
$$
|u(x,t)-u(y,s)|\leq\rho^*(|x-y|\vee\sqrt{|t-s|})
$$
for every $(x,t),(y,s)\in \overline{Q_1}$.
\end{proposition}

We will prove the aformentioned proposition in Section  \ref{sec6}. The last ingredient to be applied in the approximation step is the next lemma, which follows directly via the classical quasi-linear equation theory (see \cite[Theorem 4.4]{LSU68}) and the Schauder estimates.
\begin{lemma}
\label{lem3-8}
Let $g\in C(\partial_pQ_1)$. Let $0\leq a(x,t)\in C^1(\overline{Q_1})$. For $\varepsilon>0$, there is a unique smooth solution $u^\varepsilon\in C(\overline{Q_1})$ of \eqref{2-1} satisfying $u^\varepsilon=g$ on $\partial_pQ_1$.
\end{lemma}

With Propositions \ref{pro2-3}, \ref{pro2-4}, \ref{pro3-7} and Lemma \ref{lem3-8} in hand, we now are in a position to establish a crucial intermediate result under the assumption that $\|D_{x,t} a(x,t)\|_{L^\infty(Q_1)}$ ($\leq1$) is small, through letting $\varepsilon\rightarrow0$ in the {\em a priori} H\"{o}lder estimate in Theorem \ref{thm3-6}.

\begin{theorem}
\label{thm3-9}
Let the assumptions \eqref{1-1}--\eqref{1-3} be in force. Let $\|D_{x,t} a(x,t)\|_{L^\infty(Q_1)}\leq \kappa$, where $\kappa\in(0,1]$ is a small constant depending on $n,p,q,s,a^-,a^+$. Assume that $u$ is a bounded viscosity solution to \eqref{main} in $Q_1$. There are two constants $\alpha\in(0,1),C>0$, both of which depend on $n,p,q,s,a^-,a^+$ and $\|u\|_{L^\infty(Q_1)}$, such that the following estimates hold:
$$
|D u(x,t)-D u(y,s)|\leq C\big(|x-y|^\alpha+|t-s|^{\frac{\alpha}{2-\alpha q}}\big)
$$
and
$$
|u(x,t)-u(x,s)|\leq C|t-s|^\frac{1+\alpha}{2-\alpha q}
$$
for any $(x,t),(y,s),(x,s)\in Q_{\frac{1}{2}}$.
\end{theorem}

\begin{proof}
Given Theorem \ref{thm3-6}, Propositions \ref{pro2-3}, \ref{pro2-4}, \ref{pro3-7} and Lemma \ref{lem3-8}, the proof of this theorem is identical to that of \cite[Theorem 1]{JS17}.
\end{proof}

\subsection{H\"{o}lder regularity of spatial gradients in the case that $\|D_{x,t}a(x,t)\|_{L^\infty(Q_1)}$ is finite}

In this subsection, we prove the H\"{o}lder estimates on the gradients of solutions to equation \eqref{main} under the assumption that $D_{x,t}a(x,t)$ has a general bound instead of a small bound (less than 1). Now define
$$
\hat{u}(x,t)=\frac{1}{\epsilon}u(\epsilon x,\epsilon^2t), \quad \hat{a}(x,t)=a(\epsilon x,\epsilon^2t)
$$
with $0<\epsilon<1$. For simplicity, let $(y,s):=(\epsilon x,\epsilon^2t)$, then
\begin{align*}
&\partial_t \hat u(x,t)=\epsilon \partial_su(\epsilon x,\epsilon^2t), \\
&\partial_{x_i}\hat u(x,t)= \partial_{y_i}u(\epsilon x,\epsilon^2t), \\
&\partial_{x_ix_j}\hat u(x,t)=\epsilon \partial_{y_iy_j}u(\epsilon x,\epsilon^2t).
\end{align*}
Therefore if $u$ is a solution to \eqref{main} in $Q_1$, then we can easily check that $\hat{u}$ solves (in the viscosity sense)
\begin{equation}
\label{3-20}
\partial_t \hat{u}=[|D \hat{u}|^q+\hat{a}(x,t)|D \hat{u}|^s]\left(\delta_{ij}+(p-2)\frac{\hat{u}_i \hat{u}_j}{|D \hat{u}|^2}\right)\hat{u}_{ij}
\end{equation}
in $Q_{\epsilon^{-1}}$ and moreover
$$
\|D_{x,t}\hat{a}\|_{L^\infty(Q_{\epsilon^{-1}})}\leq \epsilon\|D_{x,t}a\|_{L^\infty(Q_1)}<\kappa
$$
by selecting
$$
\epsilon\leq\frac{\kappa}{\|D_{x,t}a\|_{L^\infty(Q_1)}+1}.
$$
By the dependencies of $\kappa$ (see Theorem \ref{thm3-9}), we know that $\epsilon$ depends only on $n,p,q,s,a^-,a^+$ and $\|D_{x,t}a\|_{L^\infty(Q_1)}$. Notice that the framework of equation \eqref{3-20} is the same as that of \eqref{main} (with $\|D_{x,t}a\|_{L^\infty(Q_1)}$ being small). Thus this allows us to make use of these results obtained above to demonstrate the interior H\"{o}lder continuity of gradients of the solutions to \eqref{3-20} and interior H\"{o}lder continuity of solutions in the time variable. In turn, by rescaling back, we can derive the local $C^{1,\alpha}$ regularity of solutions, $u$, to \eqref{main} under the condition that $\|D_{x,t}a(x,t)\|_{L^\infty(Q_1)}\leq A$. 


As has been stated above, we now conclude the proof of Theorem \ref{thm0}.

\section{Comparison principle and stability}
\label{sec4}

In this part, we are ready to show the comparison principle and stability property for the viscosity solution. When proving comparison principle, we will make use of Ishii-Lions' method. Here we consider these two properties in a more general domain. Let $\Omega$ be a bounded domain in $\mathbb{R}^n$. We denote a general parabolic cylinder by $\Omega_T:=\Omega\times [0,T)$, and $\partial_p\Omega_T$ stands for its parabolic boundary.

For the convenience of readers, here we repeat the statement before proceeding with the proof. Let $Sym(n)$ stand for the set of all symmetric $n\times n$ real matrices.

\begin{proposition}
\label{pro4-1}
Let the function $a(x,t)>0$ be Lipschitz continuous in time-space variables. Assume that $u$ and $v$ are a viscosity subsolution and a locally uniformly Lipschitz continuous viscosity supersolution in $x$-variable to \eqref{main} in $\Omega_T$, respectively. If $u\leq v$ on $\partial_p\Omega_T$, then
$$
u\leq v \quad\text{in } \Omega_T.
$$
\end{proposition}

\begin{proof}
For simplicity, we can first suppose that $v$ is a strict supersolution, that is,
$$
\partial_tv-[|D v|^q+a(x,t)|D v|^s]\left(\Delta v+(p-2)\left\langle D^2v\frac{D v}{|D v|},\frac{D v}{|D v|}\right\rangle \right)>0
$$
in the viscosity sense by considering $w:=v+\frac{\varepsilon}{T-t}$ instead. Indeed, we let $\psi\in C^2(\Omega_T)$, with $D \psi(x,t)\neq0$ for $x\neq x_0$, be such that $w-\psi$ attains a local minimum at $(x_0,t_0)\in \Omega_T$, then so does $v-\varphi$ by denoting $\varphi(x,t):=\psi(x,t)-\frac{\varepsilon}{T-t}$.
Since $v$ is a viscosity supersolution, then it yields that
\begin{align*}
0\leq &\limsup_{\stackrel{(x,t)\rightarrow(x_0,t_0)}{x\neq x_0}}\left(\partial_t \varphi(x,t)-[|D\varphi(x,t)|^q+a(x,t)|D\varphi(x,t)|^s]\Delta_p^N\varphi(x,t)\right)\\
\leq&-\frac{\varepsilon}{(T-t_0)^2}+\limsup_{\stackrel{(x,t)\rightarrow(x_0,t_0)}{x\neq x_0}}\left(\partial_t \psi(x,t)-[|D\psi(x,t)|^q+a(x,t)|D\psi(x,t)|^s]\Delta_p^N\psi(x,t)\right),
\end{align*}
and further
$$
0<\limsup_{\stackrel{(x,t)\rightarrow(x_0,t_0)}{x\neq x_0}}\left(\partial_t \psi(x,t)-[|D\psi(x,t)|^q+a(x,t)|D\psi(x,t)|^s]\Delta_p^N\psi(x,t)\right),
$$
which implies that $w$ is a strict viscosity solution by Definition \ref{def1}.

In order to show this assertion, we argue by contradiction. If the conclusion does not hold, then we may find some point $(\hat{x},\hat{t})\in \Omega\times (0,T)$ such that
$$
\omega_0:=u(\hat{x},\hat{t})-v(\hat{x},\hat{t})=\sup_{\Omega_T}(u-v)>0.
$$
Now define
$$
\Theta_j(x,y,t,s):=u(x,t)-v(y,s)-\Psi_j(x,y,t,s),
$$
where $\Psi_j(x,y,t,s)=\frac{j}{l}|x-y|^l+\frac{j}{2}(t-s)^2$ with
$$
l>\max\left\{2,\frac{q+2}{q+1},\frac{s+2}{s+1}\right\}.
$$

We denote by $(x_j,y_j,t_j,s_j)$ the maximum point of $\Theta_j$ in $\overline{\Omega}\times\overline{\Omega}\times[0,T)\times[0,T)$. It is easy to know that $(x_j,y_j,t_j,s_j)\in \Omega\times\Omega\times(0,T)\times(0,T)$ (for $j$ large enough) and $(x_j,y_j,t_j,s_j)\rightarrow (\hat{x},\hat{x},\hat{t},\hat{t})$ as $j\rightarrow \infty$ by Lemma 7.2 in \cite{CIL92}. In the rest of proof, we shall distinguish between two scenarios that $x_j=y_j$ and $x_j\neq y_j$.

\smallskip

\textbf{Case 1.} $x_j=y_j$.  Observe that, by the choice of $(x_j,y_j,t_j,s_j)$,
$$
u(x_j,t_j)-v(y_j,s_j)-\Psi_j(x_j,y_j,t_j,s_j)\geq u(x_j,t_j)-v(y,s)-\Psi_j(x_j,y,t_j,s).
$$
Let
$$
\phi(y,s):=-\Psi_j(x_j,y,t_j,s)+\Psi_j(x_j,y_j,t_j,s_j)+v(y_j,s_j).
$$
Clearly, $v(y,s)-\phi(y,s)$ has a local minimum at $(y_j,s_j)$. We first evaluate
$$
\partial_s\phi=j(t_j-s), \quad D\phi=j|x_j-y|^{l-2}(x_j-y)
$$
and
$$
D^2\phi=j|x_j-y|^{l-2}I+j(l-2)|x_j-y|^{l-2}\frac{x_j-y}{|x_j-y|}\otimes\frac{x_j-y}{|x_j-y|},
$$
where $\xi\otimes\xi$ is the matrix with entries $\xi_i\xi_j$ for a vector $\xi\in\mathbb{R}^n$. Obviously, $D\phi(y,s)\neq0$ for $y\neq x_j(=y_j)$. Owing to $v$ being a strict supersolution, we have
\begin{equation}
\label{4-1}
0<\limsup_{\stackrel{(y,s)\rightarrow(y_j,s_j)}{y\neq y_j}}\left(\partial_s \phi(y,s)-[|D\phi(y,s)|^q+a(y,s)|D\phi(y,s)|^s]\Delta_p^N\phi(y,s)\right).
\end{equation}
Next we carefully compute
\begin{align*}
&\quad\left\langle D^2\phi\frac{D \phi}{|D \phi|},\frac{D \phi}{|D \phi|}\right\rangle\\
&=\left\langle \left(j|x_j-y|^{l-2}I+j(l-2)|x_j-y|^{l-2}\frac{x_j-y}{|x_j-y|}\otimes\frac{x_j-y}{|x_j-y|}\right)\frac{x_j-y }{|x_j-y|},\frac{x_j-y}{|x_j-y|}\right\rangle\\
&=j(l-1)|x_j-y|^{l-2}
\end{align*}
and then
\begin{align*}
&[|D\phi|^q+a(y,s)|D\phi|^s]\left(\mathrm{tr}(D^2\phi)+(p-2)\left\langle D^2\phi\frac{D \phi}{|D \phi|},\frac{D \phi}{|D \phi|}\right\rangle\right)\\
=&\left[(j|x_j-y|^{l-1})^q+a(y,s)(j|x_j-y|^{l-1})^s\right]\left(j((n+l-2)+(p-2)(l-1))|x_j-y|^{l-2}\right)\\
=&\left(n+(p-2)(l-1)+l-2\right)\left[j^{q+1}|x_j-y|^{q(l-1)+l-2}+a(y,s)j^{s+1}|x_j-y|^{s(l-1)+l-2}\right],
\end{align*}
where the powers of $|x_j-y|$, $q(l-1)+l-2$ and $s(l-1)+l-2$, are positive, by the definition of $l$. From the above estimate, \eqref{4-1} turns into
$$
j(t_j-s_j)>0.
$$
On the other hand, we can see that
$$
\psi(x,t):=\Psi_j(x,y_j,t,s_j)-\Psi_j(x_j,y_j,t_j,s_j)+u(x_j,t_j)
$$
is a good testing function with respect to $u$ at $(x_j,t_j)$. In a similar way, we will readily get
\begin{align}
\label{4-2}
&j(t_j-s_j) \nonumber\\
=&\liminf_{\stackrel{(x,t)\rightarrow(x_j,t_j)}{x\neq x_j}}\left(\partial_t \psi(x,t)-[|D\psi(x,t)|^q+a(x,t)|D\psi(x,t)|^s]\Delta_p^N\psi(x,t)\right)\leq0.
\end{align}
Combining \eqref{4-1} and \eqref{4-2}, we have
$$
0<j(t_j-s_j)-j(t_j-s_j)=0,
$$
which is a contradiction.

\smallskip

\textbf{Case 2.} $x_j\neq y_j$. In this case, we shall employ the definition with jets. Applying theorem of sums (see \cite{CIL92}), for each $\mu>0$, there exist $X_j,Y_j\in Sym(n)$ such that
\begin{equation}
\label{4-3}
(\partial_t\Psi_j,D_x \Psi_j,X_j)\in\overline{\mathcal{P}}^{2,+}u(x_j,t_j),
\end{equation}
\begin{equation}
\label{4-4}
(-\partial_s\Psi_j,-D_y \Psi_j,Y_j)\in\overline{\mathcal{P}}^{2,-}v(y_j,s_j)
\end{equation}
and
\begin{equation*}
\left(\begin{array}{cc}
X_j & \\
 &-Y_j
\end{array}
\right)\leq D^2\Psi_j+\frac{1}{\mu}(D^2\Psi_j)^2,
\end{equation*}
where all the derivatives are evaluated at $(x_j,y_j,t_j,s_j)$ and
\begin{equation*}
D^2\Psi_j=\left(\begin{array}{cc}
D_{xx}\Psi_j&D_{xy}\Psi_j\\
D_{yx}\Psi_j&D_{yy}\Psi_j
\end{array}
\right)
=:\left(\begin{array}{cc}
B&-B\\
-B&B
\end{array}
\right)
\end{equation*}
with
$$
B:=j|x_j-y_j|^{l-2}I+(l-2)j|x_j-y_j|^{l-4}(x_j-y_j)\otimes(x_j-y_j).
$$
So we get
\begin{equation}
\label{4-5}
\left(\begin{array}{cc}
X_j & \\
 &-Y_j
\end{array}
\right)\leq\left(\begin{array}{cc}
B&-B\\
-B&B
\end{array}
\right)+
\frac{2}{\mu}\left(\begin{array}{cc}
B^2&-B^2\\
-B^2&B^2
\end{array}
\right),
\end{equation}
with
$$
B^2:=j^2|x_j-y_j|^{2l-4}I+l(l-2)j^2|x_j-y_j|^{2l-6}(x_j-y_j)\otimes(x_j-y_j).
$$
We plainly derive $X_j\leq Y_j$, i.e., $\langle(X_j-Y_j)\xi,\xi\rangle\leq0$ for all $\xi\in\mathbb{R}^n$. In the sequel, we choose $\mu=j$ in \eqref{4-5}. Now from \eqref{4-5}, we can get a more accurate estimate on $X_j-Y_j$ as follows:
\begin{equation}
\label{4-6}
X_j\xi\cdot\xi-Y_j\eta\cdot\eta\leq j\left[(l-1)|x_j-y_j|^{l-2}+2(l-1)^2|x_j-y_j|^{2(l-2)}\right]|\xi-\eta|^2
\end{equation}
for any $\xi,\eta\in\mathbb{R}^n$.

Next we give some notations that will be utilized later. Denote
$$
F_1(\xi,M):=|\xi|^q\left(\mathrm{tr} M+(p-2)\left\langle M\frac{\xi}{|\xi|},\frac{\xi}{|\xi|}\right\rangle\right),
$$
\begin{align*}
F_2(x,t,\xi,M)&:=a(x,t)|\xi|^s\left(\mathrm{tr} M+(p-2)\left\langle M\frac{\xi}{|\xi|},\frac{\xi}{|\xi|}\right\rangle\right)\\
&=\mathrm{tr}(A(x,t,\xi)M),
\end{align*}
where
$$
A(x,t,\xi):=a(x,t)|\xi|^s\left(I+(p-2)\frac{\xi}{|\xi|}\otimes\frac{\xi}{|\xi|}\right)
$$
with $(x,t)\in\Omega_T, M\in Sym(n)$.

Let
$$
\eta_j:=D_x \Psi_j=-D_y\Psi_j=j|x_j-y_j|^{l-2}(x_j-y_j).
$$
It is essential that $\eta_j$ is nonzero, which allows us to exploit jets. Because $u$ is a subsolution and $v$ is a strict supersolution, we arrive at
$$
-\partial_s\Psi_j-F_1(\eta_j,Y_j)-F_2(y_j,s_j,\eta_j,Y_j)>0
$$
and
$$
\partial_t\Psi_j-F_1(\eta_j,X_j)-F_2(x_j,t_j,\eta_j,X_j)\leq0
$$
by \eqref{4-3} and \eqref{4-4}. Subtracting these two inequalities above, we get
\begin{align}
\label{4-7}
0&<-\partial_s\Psi_j-\partial_t\Psi_j+F_1(\eta_j,X_j)-F_1(\eta_j,Y_j)+F_2(x_j,t_j,\eta_j,X_j)-F_2(y_j,s_j,\eta_j,Y_j) \nonumber\\
&=:J_1+J_2+J_3.
\end{align}
First, notice that
$$
-\partial_s\Psi_j=j(t_j-s_j)=\partial_t\Psi_j,
$$
then we get
$$
J_1=0.
$$
Second, through the increasing monotonicity of $F_1(\xi,M)$ with respect to the second variable $M$, and applying $X_j\leq Y_j$, we obtain
$$
J_2=F_1(\eta_j,X_j)-F_1(\eta_j,Y_j)\leq 0.
$$
In turn, we are going to estimate the third term $J_3$, which is the most delicate part of the proof. Since $a(x,t)>0$, the matrix $A(x,t,\xi)$ is positive definite so that it has matrix square root denoted by $A^\frac{1}{2}(x,t,\xi)$. By $A^\frac{1}{2}_k(x,t,\xi)$, we mean the $k$-th column of $A^\frac{1}{2}(x,t,\xi)$. Then it yields that
\begin{align}
\label{4-8}
J_3&=\mathrm{tr}(A(x_j,t_j,\eta_j)X_j)-\mathrm{tr}(A(y_j,s_j,\eta_j)Y_j) \nonumber\\
&=\sum^n_{k=1}X_jA^\frac{1}{2}_k(x_j,t_j,\eta_j)\cdot A^\frac{1}{2}_k(x_j,t_j,\eta_j)-\sum^n_{k=1}Y_jA^\frac{1}{2}_k(y_j,s_j,\eta_j)\cdot A^\frac{1}{2}_k(y_j,s_j,\eta_j)  \nonumber\\
&\leq Cj|x_j-y_j|^{l-2}\|A^\frac{1}{2}(x_j,t_j,\eta_j)-A^\frac{1}{2}(y_j,s_j,\eta_j)\|^2_2 \nonumber\\
&\leq \frac{Cj|x_j-y_j|^{l-2}}{(\lambda_{\rm min}(A^\frac{1}{2}(x_j,t_j,\eta_j))+\lambda_{\rm min}(A^\frac{1}{2}(y_j,s_j,\eta_j)))^2}
\|A(x_j,t_j,\eta_j)-A(y_j,s_j,\eta_j)\|^2_2,
\end{align}
where the penultimate inequality is derived by \eqref{4-6} and the last inequality is obtained from the local Lipschitz continuity of $M\mapsto M^\frac{1}{2}$ (see \cite[page 410]{HJ85}). Here $\lambda_{\rm min}(M)$ stands for the smallest eigenvalue of a symmetric $n\times n$ matrix $M$. Let us mention that the inequality similar to \eqref{4-8} can be found in \cite[page 1484]{JLP10}.

We proceed with evaluating
\begin{align}
\label{4-9}
&\quad\|A(x_j,t_j,\eta_j)-A(y_j,s_j,\eta_j)\|_2   \nonumber\\
&=\left\|(a(x_j,t_j)-a(y_j,s_j))|\eta_j|^s\left(I+(p-2)\frac{\eta_j}{|\eta_j|}\otimes\frac{\eta_j}{|\eta_j|}\right)\right\|_2 \nonumber\\
&\leq|\eta_j|^s|a(x_j,t_j)-a(y_j,s_j)|(\sqrt{n}+|p-2|).
\end{align}
In addition,
\begin{equation}
\label{4-10}
\begin{split}
&\lambda_{\rm min}(A^\frac{1}{2}(x_j,t_j,\eta_j))=\lambda^\frac{1}{2}_{\rm min}(A(x_j,t_j,\eta_j))\geq \min\{1,\sqrt{p-1}\}|\eta_j|^\frac{s}{2}\sqrt{a(x_j,t_j)},\\
&\lambda_{\rm min}(A^\frac{1}{2}(y_j,s_j,\eta_j))\geq\min\{1,\sqrt{p-1}\}|\eta_j|^\frac{s}{2}\sqrt{a(y_j,s_j)}.
\end{split}
\end{equation}
Merging \eqref{4-10}, \eqref{4-9} with \eqref{4-8}, we finally have
\begin{align*}
J_3
&\leq\frac{Cj|x_j-y_j|^{l-2}(\sqrt{n}+|p-2|)^2}{\min\{1,p-1\}|\eta_j|^s(\sqrt{a(x_j,t_j)}+\sqrt{a(y_j,s_j)})^2}
|\eta_j|^{2s}|a(x_j,t_j)-a(y_j,s_j)|^2\\
&\leq C\left(\sqrt{a(x_j,t_j)}+\sqrt{a(y_j,s_j)}\right)^{-2}j|x_j-y_j|^{l-2}|\eta_j|^s(|x_j-y_j|^2+|t_j-s_j|^2),
\end{align*}
where we have used the assumption that $a(x,t)$ is Lipschitz continuous in $\Omega_T$, which implies that $|a(x,t)-a(y,s)|\leq C\sqrt{|x-y|^2+|t-s|^2}$. Thus \eqref{4-7} becomes
$$
0<C\left(\sqrt{a(x_j,t_j)}+\sqrt{a(y_j,s_j)}\right)^{-2}j|x_j-y_j|^{l-2}|\eta_j|^s(|x_j-y_j|^2+|t_j-s_j|^2)=:H_j.
$$
We now verify that $H_j$ tends to 0 as $j\rightarrow\infty$, which leads to a contradiction. Next, we split the proof into two cases.

Observe that
\begin{equation*}
\begin{split}
u(x_j,t_j)-v(x_j,t_j)&\leq \max_{\overline{\Omega}\times[0,T)}\{u(x,t)-v(x,t)\}\\
&\leq u(x_j,t_j)-v(y_j,s_j)-\frac{j}{l}|x_j-y_j|^l-\frac{j}{2}(t_j-s_j)^2.
\end{split}
\end{equation*}
This leads to
\begin{align}
\label{4-12}
\frac{j}{l}|x_j-y_j|^l+\frac{j}{2}(t_j-s_j)^2&\leq v(x_j,t_j)-v(y_j,s_j)   \nonumber\\
&\rightarrow v(\hat{x},\hat{t})-v(\hat{x},\hat{t})=0,
\end{align}
as $j\rightarrow\infty$, where we have utilized the fact that $v$ is continuous in $\Omega_T$. On the other hand,
\begin{align*}
&\quad u(x_j,t_j)-v(y_j,s_j)-\frac{j}{l}|x_j-y_j|^l-\frac{j}{2}(t_j-s_j)^2\\
&=\max_{\overline{\Omega}\times\overline{\Omega}\times[0,T)\times[0,T)}
\left\{u(x,t)-v(y,s)-\frac{j}{l}|x-y|^l-\frac{j}{2}(t-s)^2\right\}\\
&\geq u(x_j,t_j)-v(x_j,s_j)-\frac{j}{l}|x_j-x_j|^l-\frac{j}{2}(t_j-s_j)^2,
\end{align*}
i.e.,
$$
v(x_j,s_j)-v(y_j,s_j)\geq\frac{j}{l}|x_j-y_j|^l.
$$
By virtue of the uniform Lipschitz continuity of $v$ in the spatial variables, we have
$$
\frac{j}{l}|x_j-y_j|^l\leq C|x_j-y_j|
$$
and further
$$
\frac{j}{l}|x_j-y_j|^{l-1}\leq C.
$$
Hence it follows that
$$
|\eta_j|\leq Cl.
$$

If $s\geq0$, then it follows from \eqref{4-12}, $l>2$ and $|\eta_j|\leq Cl$  that
$$
H_j= C\left(\sqrt{a(x_j,t_j)}+\sqrt{a(y_j,s_j)}\right)^{-2}|\eta_j|^s(j|x_j-y_j|^l+j|t_j-s_j|^2|x_j-y_j|^{l-2})\rightarrow0
$$
by sending $j\rightarrow\infty$.

If $-1<s<0$, we arrive at
\begin{align*}
H_j&=C\left(\sqrt{a(x_j,t_j)}+\sqrt{a(y_j,s_j)}\right)^{-2}|\eta_j|^{s+1}\frac{j|x_j-y_j|^{l-2}}{j|x_j-y_j|^{l-1}}(|x_j-y_j|^2+|t_j-s_j|^2)\\
&=C\left(\sqrt{a(x_j,t_j)}+\sqrt{a(y_j,s_j)}\right)^{-2}(|\eta_j|^{s+1}|x_j-y_j|+|\eta_j|^{s+1}|t_j-s_j|^2|x_j-y_j|^{-1}).
\end{align*}
Due to $|\eta_j|$ is bounded and $s+1>0$, it yields that
$$
|\eta_j|^{s+1}|x_j-y_j|\rightarrow0
$$
by $j\rightarrow\infty$, as $x_j,y_j\rightarrow\hat{x}$. Furthermore, using \eqref{4-12}, we can justify the following limit,
\begin{align*}
|\eta_j|^{s+1}|t_j-s_j|^2|x_j-y_j|^{-1}&=\left(j|x_j-y_j|^{l-1}|t_j-s_j|^\frac{2}{s+1}|x_j-y_j|^\frac{-1}{s+1}\right)^{s+1}\\
&=\left(j|t_j-s_j|^\frac{2}{s+1}|x_j-y_j|^{l-1-\frac{1}{s+1}}\right)^{s+1}\\
&\rightarrow0
\end{align*}
as $j\rightarrow\infty$, where we need to notice that $\frac{2}{s+1}>2$ by $-1<s<0$, together with $l-1-\frac{1}{s+1}>0$ by $l>\frac{s+2}{s+1}$. In conclusion, we have proved that $H_j$ converge to 0 as $j\rightarrow\infty$. The proof now is completed.
\end{proof}

We now conclude this section with stability properties of viscosity solution.

\begin{proposition}
\label{pro4-2}
Let $\{u_i\}$ be a sequence of viscosity solutions to \eqref{2-1} in $\Omega_T$ with $\varepsilon_i\geq0$ such that $\varepsilon_i\rightarrow0$. Suppose that $u_i$ converges to $u$ locally uniformly in $\Omega_T$. Then $u$ is a solution to \eqref{main} in $\Omega_T$.
\end{proposition}

\begin{proof}
We only prove that $u$ is a viscosity supersolution to \eqref{main}. The case of subsolution then follows in a similar way. Let $\varphi\in C^2(Q_1)$ be such that $u-\varphi$ reaches a local minimum at $(x_0,t_0)\in \Omega_T$ and moreover $D \varphi(x,t)\neq0$ for $x\neq x_0$. Taking into account that $u_i$ converges to $u$ locally uniformly, we can find a sequence $\{(x_i,t_i)\}\subset \Omega_T$ satisfying $(x_i,t_i)\rightarrow(x_0,t_0)$ as $i\rightarrow\infty$, such that $u_i-\varphi$ attains a local minimum at $(x_i,t_i)$. Since $u_i$ is a viscosity supersolution to \eqref{2-1}, we have
\begin{align*}
0&\leq \partial_t\varphi(x_i,t_i)-\left[(|D\varphi(x_i,t_i)|^2+\varepsilon_i^2)^\frac{q}{2}+a(x_i,t_i)(|D\varphi(x_i,t_i)|^2+\varepsilon_i^2)^\frac{s}{2}\right]\\
&\quad\cdot \left(\mathrm{tr}D^2\varphi(x_i,t_i)+(p-2)\frac{D^2\varphi(x_i,t_i)D\varphi(x_i,t_i)\cdot D\varphi(x_i,t_i)}{|D\varphi(x_i,t_i)|^2+\varepsilon_i^2}\right).
\end{align*}
Furthermore, on account of $(x_i,t_i)\rightarrow(x_0,t_0)$, we hence conclude that
\begin{equation*}
0\leq
\limsup_{\stackrel{(x,t)\rightarrow(x_0,t_0)}{x\neq x_0}}\left(\partial_t \varphi-[|D\varphi|^q+a(x,t)|D\varphi|^s]\left(\mathrm{tr}D^2\varphi+(p-2)\left\langle
D^2\varphi\frac{D \varphi}{|D \varphi|},\frac{D \varphi}{|D \varphi|}\right\rangle\right)\right),
\end{equation*}
which implies that $u$ is a viscosity supersolution to \eqref{main}.
\end{proof}

\section{Proof of Lipschitz continuity of solutions}
\label{sec5}

This section is devoted to showing the Lipschitz continuity of solutions to \eqref{2-1} with $\varepsilon\in[0,1)$, that is Lemma \ref{lem2-1}. Our proof follows roughly the similar lines as the one in \cite{Att20}. We divide the proof of Lemma \ref{lem2-1} into two steps. In the first step, we will make use of Ishii-Lions' method to infer the H\"{o}lder continuity of solutions in spatial variables. Subsequently, the H\"{o}lder continuity shall be improved into the Lipschitz continuity by employing again the Ishii-Lions' method in the second step.

To begin with, we prove the $C^{0,\gamma}$ estimates on solutions with respect to $x$-variable for all $\gamma\in(0,1)$.

\begin{lemma} [Local H\"{o}lder estimates]
\label{lem5-1}
Let the conditions \eqref{1-1} and \eqref{1-2} be in force. Let $u$ be a bounded viscosity solution to \eqref{2-1} with $\varepsilon\in[0,1)$ in $Q_1$. Assume that $a(x,t)\geq a^->0$ and $a(x,t)$ is uniformly Lipschitz continuous in $x$-variable. Then for any $\gamma\in(0,1)$, there is a positive constant $C$ depending only on $n,p,q,s,\gamma$, such that
$$
|u(x,t)-u(y,t)|\leq C\|u\|_{L^\infty(Q_1)}\left(1+\frac{C_{\rm lip}}{a^-}\right)|x-y|^\gamma
$$
for all $x,y\in B_{\frac{15}{16}}$ and $t\in\left(-\left(\frac{15}{16}\right)^2,0\right]$. Here $C_{\rm lip}$ is the same as the one in Lemma \ref{lem2-1}.
\end{lemma}

\begin{proof}
Fix $x_0,y_0\in B_{\frac{15}{16}}$ and $t_0\in\left(-\left(\frac{15}{16}\right)^2,0\right)$. We are ready to prove that there exist two suitable constants $L_1,L_2>0$ such that
$$
L:=\sup_{(x,t),(y,t)\in \overline{Q_{\frac{15}{16}}}}(u(x,t)-u(y,t)-L_1\phi(|x-y|)-\Psi(x,y,t))\leq 0,
$$
where
$$
\Psi(x,y,t)=\frac{L_2}{2}|x-x_0|^2
+\frac{L_2}{2}|y-y_0|^2+\frac{L_2}{2}|t-t_0|^2
$$
with $\phi(r):=r^\gamma$. Thriving for a contradiction. We suppose that $L>0$ and $(\overline{x},\overline{y},\overline{t})\in\overline{B_{\frac{15}{16}}}\times\overline{B_{\frac{15}{16}}}\times\left[-\left(\frac{15}{16}\right)^2,0\right]$ denotes a point reaching the maximum. By $L>0$, we know that $\overline{x}\neq\overline{y}$. Choosing
$$
L_2\geq \frac{32\|u\|_{L^\infty(Q_1)}}{(\min\{\mathrm{dist}((x_0,t_0),\partial Q_{15/16}),\mathrm{dist}((y_0,t_0),\partial Q_{15/16})\})^2},
$$
we get
$$
|\overline{x}-x_0|+|\overline{t}-t_0|\leq2\sqrt{\frac{2\|u\|_{L^\infty(Q_1)}}{L_2}}\leq\frac{\mathrm{dist}((x_0,t_0),\partial Q_{15/16})}{2}
$$
and
$$
|\overline{y}-y_0|+|\overline{t}-t_0|\leq\frac{(\mathrm{dist}(y_0,t_0),\partial Q_{15/16})}{2},
$$
so that $\overline{x},\overline{y}\in B_{15/16}$ and $\overline{t}\in\left(-\left(\frac{15}{16}\right)^2,0\right)$. In addition, if $L_1$ is large enough, we then find that
$$
|\overline{x}-\overline{y}|\leq\left(\frac{2\|u\|_{L^\infty(Q_1)}}{L_1}\right)^\frac{1}{\gamma}
$$
is sufficiently small, which is crucial and will be used later.

By Jensen-Ishii's lemma (see \cite[Theorem 8.3]{CIL92}), there are
$$
(\sigma+L_2(\overline{t}-t_0),\eta_1,X+L_2I)\in \overline{\mathcal{P}}^{2,+}u(\overline{x},\overline{t}),
$$
$$
(\sigma,\eta_2,Y-L_2I)\in \overline{\mathcal{P}}^{2,-}u(\overline{y},\overline{t}),
$$
where
\begin{align*}
&\eta_1=L_1\phi'(|\overline{x}-\overline{y}|)\frac{\overline{x}-\overline{y}}{|\overline{x}-\overline{y}|}+L_2(\overline{x}-x_0),\\
&\eta_2=L_1\phi'(|\overline{x}-\overline{y}|)\frac{\overline{x}-\overline{y}}{|\overline{x}-\overline{y}|}-L_2(\overline{y}-y_0).
\end{align*}
By choosing $L_1\geq C(\gamma)L_2$ large enough, there holds that
\begin{equation}
\label{5-1}
\frac{L_1}{2}\gamma|\overline{x}-\overline{y}|^{\gamma-1}\leq |\eta_1|,|\eta_2|\leq 2L_1\gamma|\overline{x}-\overline{y}|^{\gamma-1}.
\end{equation}
By means of Jensen-Ishii's lemma \cite[Theorem 12.2]{Cra97}, we could take $X,Y\in Sym(n)$ such that for any $\tau>0$ satisfying $\tau Z<I$, it holds that
\begin{equation}
\label{5-2}
-\frac{2}{\tau}\left(\begin{array}{cc}
I & \\
 &I
\end{array}
\right)\leq\left(\begin{array}{cc}
X&\\
&-Y
\end{array}
\right)\leq
\left(\begin{array}{cc}
Z^\tau&-Z^\tau\\
-Z^\tau&Z^\tau
\end{array}
\right),
\end{equation}
where
\begin{align*}
Z&=L_1\phi''(|\overline{x}-\overline{y}|)\frac{\overline{x}-\overline{y}}{|\overline{x}-\overline{y}|}
\otimes\frac{\overline{x}-\overline{y}}{|\overline{x}-\overline{y}|}+\frac{L_1\phi'(|\overline{x}-\overline{y}|)}{|\overline{x}-\overline{y}|}
\left(I-\frac{\overline{x}-\overline{y}}{|\overline{x}-\overline{y}|}
\otimes\frac{\overline{x}-\overline{y}}{|\overline{x}-\overline{y}|}\right)\\
&=L_1\gamma|\overline{x}-\overline{y}|^{\gamma-2}\left(I+(\gamma-2)\frac{\overline{x}-\overline{y}}{|\overline{x}-\overline{y}|}
\otimes\frac{\overline{x}-\overline{y}}{|\overline{x}-\overline{y}|}\right)
\end{align*}
and
$$
Z^\tau=(I-\tau Z)^{-1}Z.
$$
Here $(I-\tau Z)^{-1}$ stands for the inverse of the matrix $I-\tau Z$. We now pick $\tau=\frac{1}{2L_1\gamma|\overline{x}-\overline{y}|^{\gamma-2}}$ such that
$$
Z^\tau=2L_1\gamma|\overline{x}-\overline{y}|^{\gamma-2}\left(I-2\frac{2-\gamma}{3-\gamma}\frac{\overline{x}-\overline{y}}{|\overline{x}-\overline{y}|}
\otimes\frac{\overline{x}-\overline{y}}{|\overline{x}-\overline{y}|}\right).
$$
Furthermore, for $\xi=\frac{\overline{x}-\overline{y}}{|\overline{x}-\overline{y}|}$ we get
\begin{equation}
\label{5-3}
\langle Z^\tau\xi,\xi\rangle=2\gamma\frac{\gamma-1}{3-\gamma}L_1|\overline{x}-\overline{y}|^{\gamma-2}<0.
\end{equation}
It follows from \eqref{5-2} that
$$
X\leq Y
$$
and
\begin{equation}
\label{5-4}
\|X\|,\|Y\|\leq 4\gamma L_1|\overline{x}-\overline{y}|^{\gamma-2}.
\end{equation}

We next introduce a notation. Let
$$
A^\varepsilon(\eta)=I+(p-2)\frac{\eta}{(|\eta|^2+\varepsilon^2)^\frac{1}{2}}\otimes\frac{\eta}{(|\eta|^2+\varepsilon^2)^\frac{1}{2}}.
$$
It is easy to recognize that the eigenvalues of $A^\varepsilon(\eta)$ belong to $\left(\min\{1,p-1\},\max\{1,p-1\}\right)$. Since $u$ is a viscosity solution of \eqref{2-1}, we will obtain the following viscosity inequalities
$$
\sigma+L_2(\overline{t}-t_0)-\left[(|\eta_1|^2+\varepsilon^2)^\frac{q}{2}+a(\overline{x},\overline{t})(|\eta_1|^2+\varepsilon^2)^\frac{s}{2}\right]
\mathrm{tr}(A^\varepsilon(\eta_1)(X+L_2I))\leq0
$$
and
$$
\sigma-\left[(|\eta_2|^2+\varepsilon^2)^\frac{q}{2}+a(\overline{y},\overline{t})(|\eta_2|^2+\varepsilon^2)^\frac{s}{2}\right]
\mathrm{tr}(A^\varepsilon(\eta_2)(Y-L_2I))\geq0.
$$
Thus
\begin{align}
\label{5-5}
&\quad L_2(\overline{t}-t_0)  \nonumber\\
&\leq(|\eta_1|^2+\varepsilon^2)^\frac{q}{2}\mathrm{tr}(A^\varepsilon(\eta_1)(X+L_2I))-(|\eta_2|^2+\varepsilon^2)^\frac{q}{2}
\mathrm{tr}(A^\varepsilon(\eta_2)(Y-L_2I))  \nonumber\\
&\quad+a(\overline{x},\overline{t})\left[(|\eta_1|^2+\varepsilon^2)^\frac{s}{2}\mathrm{tr}(A^\varepsilon(\eta_1)(X+L_2I))
-(|\eta_2|^2+\varepsilon^2)^\frac{s}{2}\mathrm{tr}(A^\varepsilon(\eta_2)(Y-L_2I))\right]   \nonumber\\
&\quad+(a(\overline{x},\overline{t})-a(\overline{y},\overline{t}))(|\eta_2|^2+\varepsilon^2)^\frac{s}{2}\mathrm{tr}(A^\varepsilon(\eta_2)(Y-L_2I))
\nonumber\\
&=:J_1+J_2+J_3.
\end{align}
We first evaluate $J_3$ as
\begin{align*}
J_3&=(a(\overline{x},\overline{t})-a(\overline{y},\overline{t}))(|\eta_2|^2+\varepsilon^2)^\frac{s}{2}[\mathrm{tr}(A^\varepsilon(\eta_2)Y)
-L_2\mathrm{tr}(A^\varepsilon(\eta_2))]\\
&\leq C_{\rm lip}|\overline{x}-\overline{y}|C_0(s)\gamma^sL^s_1|\overline{x}-\overline{y}|^{(\gamma-1)s}[n\|A^\varepsilon(\eta_2)\|\|Y\|+
L_2(n+|p-2|)]\\
&\leq C_{\rm lip}C_0(s)\gamma^sL^s_1|\overline{x}-\overline{y}|^{(\gamma-1)s+1}[4\gamma n\max\{1,p-1\}L_1|\overline{x}-\overline{y}|^{\gamma-2}+L_2(n+|p-2|)],
\end{align*}
where we have used the inequalities \eqref{5-1} and \eqref{5-4} and the fact that $a(x,t)$ is uniformly Lipschitz continuous in $x$-variable.

We rewrite $J_1$ as
\begin{align*}
J_1&=(|\eta_1|^2+\varepsilon^2)^\frac{q}{2}\mathrm{tr}(A^\varepsilon(\eta_1)(X-Y))+(|\eta_1|^2+\varepsilon^2)^\frac{q}{2}
\mathrm{tr}((A^\varepsilon(\eta_1)-A^\varepsilon(\eta_2))Y)\\
&\quad+[(|\eta_1|^2+\varepsilon^2)^\frac{q}{2}-(|\eta_2|^2+\varepsilon^2)^\frac{q}{2}]\mathrm{tr}(A^\varepsilon(\eta_2)Y)\\
&\quad+L_2\left[(|\eta_1|^2+\varepsilon^2)^\frac{q}{2}\mathrm{tr}(A^\varepsilon(\eta_1))+(|\eta_2|^2+\varepsilon^2)^\frac{q}{2}
\mathrm{tr}(A^\varepsilon(\eta_2))\right]\\
&=:J_{1,1}+J_{1,2}+J_{1,3}+J_{1,4}.
\end{align*}
By \eqref{5-1}, it is easy to get
$$
J_{1,4}\leq 2n\gamma^qC(q)\max\{1,p-1\}L_2L^q_1|\overline{x}-\overline{y}|^{(\gamma-1)q}.
$$
In view of \eqref{5-2}, we know that all the eigenvalues of $X-Y$ are non-positive and at least one eigenvalue denoted by $\overline{\lambda}(X-Y)$ is smaller than $8\gamma\frac{\gamma-1}{3-\gamma}L_1|\overline{x}-\overline{y}|^{\gamma-2}$. Hence,
\begin{align*}
J_{1,1}&\leq (|\eta_1|^2+\varepsilon^2)^\frac{q}{2}\sum^n_{i=1}\lambda_i(A^\varepsilon(\eta_1))\lambda_i(X-Y)\\
&\leq(|\eta_1|^2+\varepsilon^2)^\frac{q}{2}\min\{1,p-1\}\overline{\lambda}(X-Y)\\
&\leq C_1(q)(\gamma L_1|\overline{x}-\overline{y}|^{\gamma-1})^q\min\{1,p-1\}8\gamma\frac{\gamma-1}{3-\gamma}L_1|\overline{x}-\overline{y}|^{\gamma-2}\\
&=8\gamma^{1+q}C_1(q)\min\{1,p-1\}\frac{\gamma-1}{3-\gamma}L^{1+q}_1|\overline{x}-\overline{y}|^{(\gamma-1)(q+1)-1}.
\end{align*}
To estimate $J_{1,2}$, we first note that
\begin{align*}
&\quad\|A^\varepsilon(\eta_1)-A^\varepsilon(\eta_2)\|\\
&=2|p-2|\left|\frac{\eta_1}{(|\eta_1|^2+\varepsilon^2)^\frac{1}{2}}-\frac{\eta_2}{(|\eta_2|^2+\varepsilon^2)^\frac{1}{2}}\right|\\
&\leq2|p-2|\left(\left|\frac{\eta_1}{|\eta_1|}-\frac{\eta_2}{|\eta_2|}\right|\frac{|\eta_1|}{(|\eta_1|^2+\varepsilon^2)^\frac{1}{2}}
+\left|\frac{|\eta_1|}{(|\eta_1|^2+\varepsilon^2)^\frac{1}{2}}-\frac{|\eta_2|}{(|\eta_2|^2+\varepsilon^2)^\frac{1}{2}}\right|\right)\\
&\leq4|p-2|\max\left\{\frac{|\eta_1-\eta_2|}{|\eta_1|},\frac{|\eta_1-\eta_2|}{|\eta_2|}\right\}\\
&\leq\frac{32|p-2|L_2}{\gamma L_1|\overline{x}-\overline{y}|^{\gamma-1}},
\end{align*}
where $|\eta_1-\eta_2|\leq 4L_2$. Thus via \eqref{5-1} and \eqref{5-4} we arrive at
\begin{align*}
J_{1,2}&\leq(|\eta_1|^2+\varepsilon^2)^\frac{q}{2}n\|A^\varepsilon(\eta_1)-A^\varepsilon(\eta_2)\|\|Y\|\\
&\leq C_2(q)(\gamma L_1|\overline{x}-\overline{y}|^{\gamma-1})^qn\frac{32|p-2|L_2}{\gamma L_1|\overline{x}-\overline{y}|^{\gamma-1}}4\gamma L_1|\overline{x}-\overline{y}|^{\gamma-2}\\
&=128nC_2(q)\gamma^q|p-2|L_2L_1^q|\overline{x}-\overline{y}|^{(\gamma-1)q-1}.
\end{align*}
We finally estimate $J_{1,3}$. Applying the mean value theorem and \eqref{5-1}, we evaluate
\begin{align*}
&\quad|(|\eta_1|^2+\varepsilon^2)^\frac{q}{2}-(|\eta_2|^2+\varepsilon^2)^\frac{q}{2}|\\
&=\frac{q}{2}\zeta^{\frac{q}{2}-1}||\eta_1|^2-|\eta_2|^2|\\
&\leq C_3(q)(\gamma L_1|\overline{x}-\overline{y}|^{\gamma-1})^{q-2}(\gamma L_1|\overline{x}-\overline{y}|^{\gamma-1})||\eta_1|-|\eta_2||\\
&\leq C_3(q)\gamma^{q-1}L_1^{q-1}|\overline{x}-\overline{y}|^{(q-1)(\gamma-1)}|\eta_1-\eta_2|\\
&\leq4C_3(q)\gamma^{q-1}L_2 L_1^{q-1}|\overline{x}-\overline{y}|^{(q-1)(\gamma-1)},
\end{align*}
where $\zeta$ is between $|\eta_1|^2+\varepsilon^2$ and $|\eta_2|^2+\varepsilon^2$. Then it follows from \eqref{5-4} that
\begin{align*}
J_{1,3}&\leq4C_3(q)\gamma^{q-1}L_2 L_1^{q-1}|\overline{x}-\overline{y}|^{(q-1)(\gamma-1)}n\|Y\|\|A^\varepsilon(\eta_2)\|\\
&\leq16nC_3(q)\max\{1,p-1\}\gamma^qL_2 L_1^q|\overline{x}-\overline{y}|^{q(\gamma-1)-1}.
\end{align*}
Combining the estimates on $J_{1,1}$, $J_{1,2}$, $J_{1,3}$ and $J_{1,4}$, we derive
\begin{align*}
J_1\leq&-8\gamma^{1+q}C_1(q)\min\{1,p-1\}\frac{1-\gamma}{3-\gamma}L^{1+q}_1|\overline{x}-\overline{y}|^{(\gamma-1)(q+1)-1}
\\
&+128nC_2(q)\gamma^q|p-2|L_2L_1^q|\overline{x}-\overline{y}|^{(\gamma-1)q-1}\\
&+16nC_3(q)\max\{1,p-1\}\gamma^qL_2 L_1^q|\overline{x}-\overline{y}|^{q(\gamma-1)-1}\\
&+2n\gamma^qC(q)\max\{1,p-1\}L_2L^q_1|\overline{x}-\overline{y}|^{(\gamma-1)q}.
\end{align*}
Analogously, we can arrive at
\begin{align*}
J_2&\leq a(\overline{x},\overline{t})\big[-8\gamma^{1+s}C_1(s)\min\{1,p-1\}\frac{1-\gamma}{3-\gamma}L^{1+s}_1|\overline{x}-\overline{y}|^{(\gamma-1)(s+1)-1}
\\
&\quad+128nC_2(s)\gamma^s|p-2|L_2L_1^s|\overline{x}-\overline{y}|^{(\gamma-1)s-1}\\
&\quad+16nC_3(s)\max\{1,p-1\}\gamma^sL_2 L_1^s|\overline{x}-\overline{y}|^{(\gamma-1)s-1}\\
&\quad+2n\gamma^sC(s)\max\{1,p-1\}L_2L^s_1|\overline{x}-\overline{y}|^{(\gamma-1)s}\big].
\end{align*}
Because the constant coefficients in $J_1, J_2, J_3$ are too long, we simply denote them by $C$, possibly varying from line to line. And moreover relevant dependencies on parameters will be emphasised using parentheses. Consequently, \eqref{5-5} becomes
\begin{align*}
0&\leq L_2+\big[-8C(p,q,\gamma)L^{1+q}_1|\overline{x}-\overline{y}|^{(\gamma-1)(q+1)-1}
+C(n,p,q,\gamma)L_2L_1^q|\overline{x}-\overline{y}|^{(\gamma-1)q-1}\\
&\quad+C(n,p,q,\gamma)L_2 L_1^q|\overline{x}-\overline{y}|^{q(\gamma-1)-1}
+C(n,p,q,\gamma)L_2L^q_1|\overline{x}-\overline{y}|^{(\gamma-1)q}\big]\\
&\quad+a(\overline{x},\overline{t})\big[-8C(p,s,\gamma)L^{1+s}_1|\overline{x}-\overline{y}|^{(\gamma-1)(s+1)-1}
+C(n,p,s,\gamma)L_2L_1^s|\overline{x}-\overline{y}|^{(\gamma-1)s-1}\\
&\quad+C(n,p,s,\gamma)L_2 L_1^s|\overline{x}-\overline{y}|^{(\gamma-1)s-1}
+C(n,p,s,\gamma)L_2L^s_1|\overline{x}-\overline{y}|^{(\gamma-1)s}\big]\\
&\quad+C_{\rm lip}\big[C(n,p,s,\gamma)L_1^{s+1}|\overline{x}-\overline{y}|^{(\gamma-1)(s+1)}
+C(n,p,s,\gamma)L_2L_1^s|\overline{x}-\overline{y}|^{(\gamma-1)s}\big].
\end{align*}
We rearrange the previous display as
\begin{align*}
0&\leq\big[-8C(p,q,\gamma)L^{1+q}_1|\overline{x}-\overline{y}|^{(\gamma-1)(q+1)-1}
+C(n,p,q,\gamma)L_2L_1^q|\overline{x}-\overline{y}|^{(\gamma-1)q-1}\big]\\
&\quad+a(\overline{x},\overline{t})\big[-8C(p,s,\gamma)L^{1+s}_1|\overline{x}-\overline{y}|^{(\gamma-1)(s+1)-1}
+C(n,p,s,\gamma)L_2L_1^s|\overline{x}-\overline{y}|^{(\gamma-1)s-1}\\
&\quad+a(\overline{x},\overline{t})^{-1}C_{\rm lip}C(n,p,s,\gamma)L_1^{s+1}|\overline{x}-\overline{y}|^{(\gamma-1)(s+1)}\\
&\quad+a(\overline{x},\overline{t})^{-1}C_{\rm lip}C(n,p,s,\gamma)L_2L_1^s|\overline{x}-\overline{y}|^{(\gamma-1)s-1}\big].
\end{align*}
Therefore, we can select $L_1$ large enough such that
\begin{equation*}
\begin{cases}
L_1|\overline{x}-\overline{y}|^{\gamma-1}\geq C(n,p,q,\gamma)L_2,\\[2mm]
|\overline{x}-\overline{y}|^{-1}\geq C(n,p,s,\gamma)(a^-)^{-1}C_{\rm lip},\\[2mm]
L_1|\overline{x}-\overline{y}|^{\gamma-1}\geq C(n,p,s,\gamma)L_2,\\[2mm]
L_1|\overline{x}-\overline{y}|^{\gamma-1}\geq C(n,p,s,\gamma)(a^-)^{-1}C_{\rm lip}L_2.
\end{cases}
\end{equation*}
Thanks to $|\overline{x}-\overline{y}|\leq\left(\frac{2\|u\|_{L^\infty(Q_1)}}{L_1}\right)^\frac{1}{\gamma}$, then we require that
$$
\left(\frac{2\|u\|_{L^\infty(Q_1)}}{L_1}\right)^{-\frac{1}{\gamma}}\geq C(n,p,s,\gamma)(a^-)^{-1}C_{\rm lip}
$$
so that $|\overline{x}-\overline{y}|^{-1}\geq C(n,p,s,\gamma)(a^-)^{-1}C_{\rm lip}$ holds true. As a consequence, we can fix
$$
L_1=C(n,p,q,s,\gamma)\left(\|u\|_{L^\infty(Q_1)}+(a^-)^{-1}C_{\rm lip}\|u\|_{L^\infty(Q_1)}\right),
$$
then derive that
$$
0\leq-7C(p,q,\gamma)L^{1+q}_1|\overline{x}-\overline{y}|^{(\gamma-1)(q+1)-1}-a^-5C(p,s,\gamma)
L^{1+s}_1|\overline{x}-\overline{y}|^{(\gamma-1)(s+1)-1}.
$$
That is a contradiction. Finally, we obtain the desired result by the dependence of $L_1$.
\end{proof}

Based on Lemma \ref{lem5-1}, we can demonstrate the Lipschitz continuity (Lemma \ref{lem2-1}) of solutions to \eqref{2-1} in the spatial variables via applying Ishii-Lions methods again.

\medskip

\noindent\textbf{Proof of Lemma \ref{lem2-1}.} Fix $r=\frac{7}{8}$ and $x_0,y_0\in B_r$, $t_0\in(-r^2,0)$. We are going to show that there exist two suitable constants $M_1,M_2>0$ such that
$$
M_0:=\sup_{(x,t),(y,t)\in \overline{Q_{\frac{7}{8}}}}(u(x,t)-u(y,t)-M_1\varphi(|x-y|)-\Phi(x,y,t))\leq 0,
$$
where
$$
\Phi(x,y,t)=\frac{M_2}{2}|x-x_0|^2
+\frac{M_2}{2}|y-y_0|^2+\frac{M_2}{2}|t-t_0|^2
$$
and
\begin{equation*}
\varphi(r):=\begin{cases}r-\kappa_0r^\nu  & {0\leq r\leq r_1:=\left(\frac{1}{4\nu\kappa_0}\right)^\frac{1}{\nu-1}},\\[2mm]
\varphi(r_1)  & {r>r_1}
\end{cases}
\end{equation*}
with $1<\nu<2$ and $0<\kappa_0<1$ such that $2<r_1$. Observe that, for $r\in(0,r_1)$,
\begin{equation*}
\begin{cases}
\varphi'(r)=1-\nu\kappa_0r^{\nu-1},\\[2mm]
\varphi''(r)=-\nu(\nu-1)\kappa_0r^{\nu-2},
\end{cases}
\end{equation*}
and then $\frac{3}{4}\leq \varphi'(r)\leq1$ and $\varphi''(r)<0$ when $r\in(0,2]$. We now argue by contradiction. Assume that $M_0>0$ and $(\hat{x},\hat{y},\hat{t})\in \overline{B_r}\times\overline{B_r}\times[-r^2,0]$ represents a point attaining the maximum. As in the proof of Lemma \ref{lem5-1}, we recognize that $\hat{x}\neq\hat{y}$ and $\hat{x},\hat{y}\in B_r$, $\hat{t}\in(-r^2,0)$ for $M_2\geq C\|u\|_{L^\infty(Q_1)}$. Furthermore, we know from Lemma \ref{lem5-1} that $u$ is locally H\"{o}lder continuous in $x$-variable and for some $\gamma\in(0,1)$, it holds that
$$
|u(x,t)-u(y,t)|\leq C\|u\|_{L^\infty(Q_1)}\left[1+(a^-)^{-1}C_{\rm lip}\right]|x-y|^\gamma.
$$
In the rest of proof, for shortness we denote
$$
K:= C\|u\|_{L^\infty(Q_1)}\left[1+(a^-)^{-1}C_{\rm lip}\right].
$$
Employing the above inequality and choosing $2M_2\leq K$, we get
\begin{equation}
\label{5-6}
M_2|\hat{y}-y_0|,M_2|\hat{x}-x_0|\leq K|\hat{x}-\hat{y}|^\frac{\gamma}{2}.
\end{equation}
In addition, it follows from $M_0>0$ that
$$
M_1(|\hat{x}-\hat{y}|-\kappa_0|\hat{x}-\hat{y}|^\nu)\leq 2\|u\|_{L^\infty(Q_1)},
$$
i.e.,
$$
M_1|\hat{x}-\hat{y}|(1-\kappa_0|\hat{x}-\hat{y}|^{\nu-1})\leq 2\|u\|_{L^\infty(Q_1)}.
$$
Thus we can fix $0<\kappa_0<1$ such that $\frac{1}{2}\leq1-\kappa_0|\hat{x}-\hat{y}|^{\nu-1}$, from which we notice that
$$
|\hat{x}-\hat{y}|\leq\frac{4\|u\|_{L^\infty(Q_1)}}{M_1}.
$$ 

By Jensen-Ishii's lemma, we have
$$
(\sigma+M_2(\hat{t}-t_0),\eta_1,X+M_2I)\in \overline{\mathcal{P}}^{2,+}u(\hat{x},\hat{t}),
$$
$$
(\sigma,\eta_2,Y-M_2I)\in \overline{\mathcal{P}}^{2,-}u(\hat{y},\hat{t}),
$$
and for any $\tau>0$,
\begin{equation}
\label{5-7}
-(\tau+2\|Z\|)\left(\begin{array}{cc}
I & \\
 &I
\end{array}
\right)\leq\left(\begin{array}{cc}
X&\\
&-Y
\end{array}
\right)\leq
\left(\begin{array}{cc}
Z&-Z\\
-Z&Z
\end{array}
\right)+\frac{2}{\tau}\left(\begin{array}{cc}
Z^2&-Z^2\\
-Z^2&Z^2
\end{array}
\right),
\end{equation}
where
\begin{align*}
&\eta_1=M_1\varphi'(|\hat{x}-\hat{y}|)\frac{\hat{x}-\hat{y}}{|\hat{x}-\hat{y}|}+M_2(\hat{x}-x_0),\\
&\eta_2=M_1\varphi'(|\hat{x}-\hat{y}|)\frac{\hat{x}-\hat{y}}{|\hat{x}-\hat{y}|}-M_2(\hat{y}-y_0)
\end{align*}
and
$$
Z=M_1\varphi''(|\hat{x}-\hat{y}|)\frac{\hat{x}-\hat{y}}{|\hat{x}-\hat{y}|}
\otimes\frac{\hat{x}-\hat{y}}{|\hat{x}-\hat{y}|}+M_1\frac{\varphi'(|\hat{x}-\hat{y}|)}{|\hat{x}-\hat{y}|}
\left(I-\frac{\hat{x}-\hat{y}}{|\hat{x}-\hat{y}|}
\otimes\frac{\hat{x}-\hat{y}}{|\hat{x}-\hat{y}|}\right).
$$
Direct computations give that
\begin{equation}
\label{5-8}
\frac{M_1}{2}\leq|\eta_1|,|\eta_2|\leq2M_1, \quad \text{if } M_1\geq4K
\end{equation}
and
\begin{equation}
\label{5-9}
\|Z\|\leq M_1\frac{\varphi'(|\hat{x}-\hat{y}|)}{|\hat{x}-\hat{y}|},
\end{equation}
\begin{equation}
\label{5-10}
\|Z^2\|\leq M_1^2\left(|\varphi''(|\hat{x}-\hat{y}|)|+\frac{\varphi'(|\hat{x}-\hat{y}|)}{|\hat{x}-\hat{y}|}\right)^2.
\end{equation}
According to \eqref{5-7}, we infer that $X-Y\leq0$. Besides, by taking
$$
\tau=4M_1\left(|\varphi''(|\hat{x}-\hat{y}|)|+\frac{\varphi'(|\hat{x}-\hat{y}|)}{|\hat{x}-\hat{y}|}\right)
$$
and applying \eqref{5-7} to the vector $(\xi,-\xi)$ with $\xi=\frac{\hat{x}-\hat{y}}{|\hat{x}-\hat{y}|}$, after careful manipulation, we derive that
\begin{equation}
\label{5-11}
\langle(X-Y)\xi,\xi\rangle\leq 4\left(\langle Z\xi,\xi\rangle+\frac{2}{\tau}\langle Z^2\xi,\xi\rangle\right)\leq 2M_1\varphi''(|\hat{x}-\hat{y}|),
\end{equation}
which indicates that at least one eigenvalue of $X-Y$ denoted by $\overline{\lambda}(X-Y)$ is smaller than $2M_1\varphi''(|\hat{x}-\hat{y}|)<0$. Writing the viscosity inequalities and adding them, we arrive at
\begin{align}
\label{5-12}
&\quad M_2(\hat{t}-t_0)  \nonumber\\
&\leq(|\eta_1|^2+\varepsilon^2)^\frac{q}{2}\mathrm{tr}(A^\varepsilon(\eta_1)(X+M_2I))-(|\eta_2|^2+\varepsilon^2)^\frac{q}{2}
\mathrm{tr}(A^\varepsilon(\eta_2)(Y-M_2I))  \nonumber\\
&\quad+a(\hat{x},\hat{t})\left[(|\eta_1|^2+\varepsilon^2)^\frac{s}{2}\mathrm{tr}(A^\varepsilon(\eta_1)(X+M_2I))
-(|\eta_2|^2+\varepsilon^2)^\frac{s}{2}\mathrm{tr}(A^\varepsilon(\eta_2)(Y-M_2I))\right]   \nonumber\\
&\quad+(a(\hat{x},\hat{t})-a(\hat{y},\hat{t}))(|\eta_2|^2+\varepsilon^2)^\frac{s}{2}\mathrm{tr}(A^\varepsilon(\eta_2)(Y-M_2I))
\nonumber\\
&=:J_1+J_2+J_3.
\end{align}
Putting together \eqref{5-7}, \eqref{5-9} and \eqref{5-10}, we deduce that
\begin{align}
\label{5-13}
\|Y\| &\leq 2|\langle Z\overline{\xi},\overline{\xi}\rangle|+\frac{4}{\tau}|\langle Z^2\overline{\xi},\overline{\xi}\rangle|  \nonumber\\
&\leq4M_1\left(|\varphi''(|\hat{x}-\hat{y}|)|+\frac{\varphi'(|\hat{x}-\hat{y}|)}{|\hat{x}-\hat{y}|}\right),
\end{align}
where $\overline{\xi}$ is a unit vector. In what follows, we estimate the terms $J_1,J_2,J_3$ in a similar way to that in proof of Lemma \ref{lem5-1}. So we write it briefly. It yields by \eqref{5-8} and \eqref{5-11} that
\begin{equation}
\label{5-14}
(|\eta_1|^2+\varepsilon^2)^\frac{q}{2}\mathrm{tr}(A^\varepsilon(\eta_1)(X-Y))\leq2\min\{1,1-p\}C_1(q)M_1^{1+q}\varphi''(|\hat{x}-\hat{y}|).
\end{equation}
In view of \eqref{5-6}, \eqref{5-8}, \eqref{5-13}, we have
\begin{align}
\label{5-15}
&\quad(|\eta_1|^2+\varepsilon^2)^\frac{q}{2}
\mathrm{tr}((A^\varepsilon(\eta_1)-A^\varepsilon(\eta_2))Y)  \nonumber\\
&\leq64nC_2(q)|p-2|KM^q_1\left(\varphi'(|\hat{x}-\hat{y}|)|\hat{x}-\hat{y}|^{\frac{\gamma}{2}-1}+|\varphi''(|\hat{x}-\hat{y}|)|\right).
\end{align}
By the mean value theorem together with the inequalities \eqref{5-6} and \eqref{5-8}, it follows that
$$
|(|\eta_1|^2+\varepsilon^2)^\frac{q}{2}-(|\eta_2|^2+\varepsilon^2)^\frac{q}{2}|\leq C_3(q)KM^{q-1}_1|\hat{x}-\hat{y}|^\frac{\gamma}{2},
$$
which leads to
\begin{align}
\label{5-16}
&\quad|(|\eta_1|^2+\varepsilon^2)^\frac{q}{2}-(|\eta_2|^2+\varepsilon^2)^\frac{q}{2}||\mathrm{tr}(A^\varepsilon(\eta_2)Y)| \nonumber\\
&\leq nC_3(q)\max\{1,p-1\}KM^q_1\left(|\hat{x}-\hat{y}|^{\frac{\gamma}{2}-1}+|\varphi''(|\hat{x}-\hat{y}|)|\right).
\end{align}
We evaluate
\begin{equation}
\label{5-17}
\begin{split}
&\quad M_2\left[(|\eta_1|^2+\varepsilon^2)^\frac{q}{2}\mathrm{tr}(A^\varepsilon(\eta_1))+(|\eta_2|^2+\varepsilon^2)^\frac{q}{2}
\mathrm{tr}(A^\varepsilon(\eta_2))\right]\\
&\leq2^{q+1}n\max\{1,p-1\}M_2M^q_1.
\end{split}
\end{equation}
Merging the estimates \eqref{5-14}--\eqref{5-17}, we obtain
\begin{align*}
J_1
&\leq2C_1(q)\min\{1,p-1\}M^{1+q}_1\varphi''(|\hat{x}-\hat{y}|)\\
&\quad+64n
C_2(q)|p-2|KM_1^q\left(|\hat{x}-\hat{y}|^{\frac{\gamma}{2}-1}+|\varphi''(|\hat{x}-\hat{y}|)|\right)\\
&\quad+nC_3(q)\max\{1,p-1\}KM_1^q\left(|\hat{x}-\hat{y}|^{\frac{\gamma}{2}-1}+|\varphi''(|\hat{x}-\hat{y}|)|\right)\\
&\quad+2^{q+1}n\max\{1,p-1\}M_2M^q_1.
\end{align*}
Likewise, we can arrive at
\begin{align*}
J_2&\leq a(\hat{x},\hat{t})\big[2C_1(s)\min\{1,p-1\}M^{1+s}_1\varphi''(|\hat{x}-\hat{y}|)\\
&\quad+64n
C_2(s)|p-2|KM_1^s(|\hat{x}-\hat{y}|^{\frac{\gamma}{2}-1}+|\varphi''(|\hat{x}-\hat{y}|)|)\\
&\quad+nC_3(s)\max\{1,p-1\}KM_1^s(|\hat{x}-\hat{y}|^{\frac{\gamma}{2}-1}+|\varphi''(|\hat{x}-\hat{y}|)|)\\
&\quad+2^{s+1}n\max\{1,p-1\}M_2M^s_1\big].
\end{align*}

Finally, $J_3$ could be evaluated as
$$
J_3\leq 4nC_4(s)\max\{1,p-1\}C_{\rm lip}M^s_1(M_1(1+|\varphi''(|\hat{x}-\hat{y}|)||\hat{x}-\hat{y}|)+M_2),
$$
where we have used the fact that $|a(\hat{x},\hat{t})-a(\hat{y},\hat{t})|\leq C_{\rm lip}|\hat{x}-\hat{y}|$. Observe that $\varphi''(|\hat{x}-\hat{y}|)=-\nu(\nu-1)\kappa_0|\hat{x}-\hat{y}|^{\nu-2}$. Because the constant coefficients in $J_1, J_2, J_3$ are too long, we simply denote by $C$ as before. And moreover relevant dependencies on parameters will be emphasised using parentheses. Next, we take $\nu=\frac{\gamma}{2}+1$. Therefore, after rearrangement inequality \eqref{5-12} turns into
\begin{align}
\label{5-18}
0&\leq\big[-2C(p,q)M^{1+q}_1|\hat{x}-\hat{y}|^{\frac{\gamma}{2}-1}+C(n,p,q)KM^q_1|\hat{x}-\hat{y}|^{\frac{\gamma}{2}-1}+C(n,p,q)M_2M_1^q\big] \nonumber\\
&\quad+a(\hat{x},\hat{t})\big[-2C(p,s)M^{1+s}_1|\hat{x}-\hat{y}|^{\frac{\gamma}{2}-1}+C(n,p,s)KM_1^s|\hat{x}-\hat{y}|^{\frac{\gamma}{2}-1}+C(n,p,s)M_2M^s_1 \nonumber\\
&\quad+a(\hat{x},\hat{t})^{-1}C(n,p,s)C_{\rm lip}M_1^{1+s}+a(\hat{x},\hat{t})^{-1}C(n,p,s)C_{\rm lip}M_2M_1^s\big].
\end{align}
To reach a contradiction, we shall choose $M_1$ large so that
\begin{equation*}
\begin{cases}
\frac{1}{2}C(p,q)M^{1+q}_1|\hat{x}-\hat{y}|^{\frac{\gamma}{2}-1}\geq C(n,p,q)KM_1^q|\hat{x}-\hat{y}|^{\frac{\gamma}{2}-1},\\[2mm]
\frac{1}{2}C(p,q)M^{1+q}_1|\hat{x}-\hat{y}|^{\frac{\gamma}{2}-1}\geq C(n,p,q)M_2M_1^q,\\[2mm]
\frac{1}{4}C(p,s)M^{1+s}_1|\hat{x}-\hat{y}|^{\frac{\gamma}{2}-1}\geq C(n,p,s)(a^-)^{-1}C_{\rm lip}M^{1+s}_1,\\[2mm]
\frac{1}{4}C(p,s)M^{1+s}_1|\hat{x}-\hat{y}|^{\frac{\gamma}{2}-1}\geq C(n,p,s)KM_1^s|\hat{x}-\hat{y}|^{\frac{\gamma}{2}-1},\\[2mm]
\frac{1}{4}C(p,s)M^{1+s}_1|\hat{x}-\hat{y}|^{\frac{\gamma}{2}-1}\geq C(n,p,s)M_2M^s_1,\\[2mm]
\frac{1}{4}C(p,s)M^{1+s}_1|\hat{x}-\hat{y}|^{\frac{\gamma}{2}-1}\geq C(n,p,s)(a^-)^{-1}C_{\rm lip}M_2M_1^s.
\end{cases}
\end{equation*}
Utilizing the known fact that $|\hat{x}-\hat{y}|\leq\frac{4\|u\|_{L^\infty(Q_1)}}{M_1}$, we can take
$$
M_1\geq C(n,p,q,s)\left(\|u\|_{L^\infty(Q_1)}\left(1+(a^-)^{-1}C_{\rm lip}+[(a^-)^{-1}C_{\rm lip}]^\frac{2}{2-\gamma}\right)+K\right)
$$
to satisfy the above requirement. After manipulation, we can eventually fix
$$
M_1=C(n,p,q,s)\|u\|_{L^\infty(Q_1)}\left(1+[(a^-)^{-1}C_{\rm lip}]^2\right)
$$
such that display \eqref{5-18} becomes
$$
0\leq -C(p,q)M^{1+q}_1|\hat{x}-\hat{y}|^{\frac{\gamma}{2}-1}-a(\hat{x},\hat{t})C(p,s)M^{1+s}_1|\hat{x}-\hat{y}|^{\frac{\gamma}{2}-1}.
$$
That is a contradiction.  Therefore, we obtain the desired result by the dependence of $M_1$.

\section{The proof of Proposition \ref{pro3-7}}
\label{sec6}

In this section, we will give the proof of the boundary estimates in Proposition \ref{pro3-7}. For the sake of convenience, denote
$$
F_\varepsilon(x,t,D u,D^2u)=\big[(|D u|^2+\varepsilon^2)^\frac{q}{2}+a(x,t)(|D u|^2+\varepsilon^2)^\frac{s}{2}\big]\left(\delta_{ij}+(p-2)\frac{u_iu_j}{|D u|^2+\varepsilon^2}\right)u_{ij}.
$$
In what follows, we assume that the conditions \eqref{1-1}--\eqref{1-3} are in force.

\begin{lemma}
\label{lem6-1}
For each $y\in\partial B_1$, there is a function $W_y(x)\in C(\overline{B_1})$ fulfilling $W_y(y)=0$ and $W_y(x)>0$ in $\overline{B_1}\setminus\{y\}$, and for all $t\in[-1,0]$
$$
F_\varepsilon(x,t,D W_y,D^2W_y)\leq-1 \quad\text{in } B_1.
$$
\end{lemma}

\begin{proof}
Let $y\in\partial B_1$, $f(r)=\sqrt{(r-1)_+}$ and $g_y(x)=f(|x-2y|)$. Here $(r-1)_+:=\max\{r-1,0\}$. By direct calculation, then for $x\in B_1$ and $t\in[-1,0]$, we derive
\begin{align*}
&\quad F_\varepsilon(x,t,D g_y,D^2g_y)\\
&=\left[(f'^2+\varepsilon^2)^\frac{q}{2}+a(x,t)(f'^2+\varepsilon^2)^\frac{s}{2}\right]\left[\left(1+(p-2)
\frac{f'^2}{f'^2+\varepsilon^2}\right)f''+\frac{n-1}{|x-2y|}f'\right]\\
&\leq\left[(f'^2+\varepsilon^2)^\frac{q}{2}+a(x,t)(f'^2+\varepsilon^2)^\frac{s}{2}\right](r-1)^{-\frac{1}{2}}\left(
\frac{n-1}{2}-\frac{\min\{1,p-1\}}{4}(r-1)^{-1}\right)\\
&\leq\left(\frac{(r-1)^{-1}}{4}+\varepsilon^2\right)^\frac{q}{2}(r-1)^{-\frac{1}{2}}\left(
\frac{n-1}{2}-\frac{\min\{1,p-1\}}{4}(r-1)^{-1}\right),
\end{align*}
if we choose $r>1$ sufficiently close to 1, where $r:=|x-2y|$. Therefore, there exists $\delta>0$, that depends only on $n,p,q$, such that for $x\in B_1\cap B_{1+\delta}(2y)$ and $t\in[-1,0]$ it holds that
$$
F_\varepsilon(x,t,D g_y,D^2g_y)\leq-1.
$$
Set
$$
h_y(x)=b\left(2^\sigma-\frac{1}{|x-2y|^\sigma}\right)
$$
with
$$
\sigma=\frac{2n}{\min\{1,p-1\}}+2 \quad\text{and} \quad b>0.
$$
Clearly, $h_y(x)\geq b(2^\sigma-1)$ in $B_1$. Also, for $x\in B_1$ and $t\in[-1,0]$, using the choice of $\sigma$ and $1<r<3$, we get
\begin{align*}
&\quad F_\varepsilon(x,t,D h_y,D^2h_y)\\
&=b\left[(b^2\sigma^2r^{-2\sigma-2}+\varepsilon^2)^\frac{q}{2}+a(x,t)(b^2\sigma^2r^{-2\sigma-2}+\varepsilon^2)^\frac{s}{2}\right]\\
&\quad\cdot\left[\left(1+\frac{(p-2)\sigma^2}{\sigma^2+\varepsilon^2b^{-2}r^{2\sigma+2}}\right)\sigma(-\sigma-1)r^{-\sigma-2}+(n-1)\sigma r^{-\sigma-2}\right]\\
&\leq b\left[(b^2\sigma^2r^{-2\sigma-2}+\varepsilon^2)^\frac{q}{2}+a(x,t)(b^2\sigma^2r^{-2\sigma-2}+\varepsilon^2)^\frac{s}{2}\right]
\left(-\frac{1}{2}\sigma r^{-\sigma-2}\right)\\
&\leq-\frac{b}{2}\sigma r^{-\sigma-2}(b^2\sigma^2r^{-2\sigma-2}+\varepsilon^2)^\frac{q}{2}\\
&\leq\begin{cases}-\frac{b^{1+q}}{2}3^{-\sigma-2-q(\sigma+1)}\sigma^{1+q} &{ \text{if } q\geq 0},\\[2mm]
-\frac{b}{2}3^{-\sigma-2}(1+\sigma^2)^\frac{q}{2}\sigma &{\text{if } -1<q<0}.
\end{cases}
\end{align*}
Then we take $b<1$ satisfying
$$
b\left(2^\sigma-\frac{1}{|1+\delta|^\sigma}\right)=\sqrt{\frac{\delta}{2}}.
$$
Thanks to $g_y(y)=0$ and $h_y(y)>0$, the function
\begin{equation*}
W_y(x)=\begin{cases} h_y(x)   &{ \text{if } x\in \overline{B_1}, |x-2y|\geq 1+\delta},\\[2mm]
\min\{g_y(x),h_y(x)\}   &{ \text{if } x\in \overline{B_1}, |x-2y|\leq 1+\delta}
\end{cases}
\end{equation*}
consists with $g_y$ in a neighborhood of $y$, and, by the selection of $b$, consists with $h_y$ when $x\in \overline{B_1}$ and $|x-2y|\geq 1+\tilde{\delta}$ for some $\tilde{\delta}\in(0,\delta)$. Furthermore,
$$
F_\varepsilon(x,t,D W_y,D^2W_y)\leq-C
$$
for $x\in B_1$ and $t\in[-1,0]$, where $C>0$ depends only upon $n,p$ and $q$. We conclude this proof through multiplying a large positive constant to $W_y$.
\end{proof}

\begin{lemma}
\label{lem6-2}
For each $(y,\tau)\in \partial_pQ_1$, there is $W_{y,\tau}\in C(\overline{Q_1})$ satisfying $W_{y,\tau}(y,\tau)=0$, $W_{y,\tau}>0$ in $\overline{Q_1}\setminus\{(y,\tau)\}$ as well as
$$
\partial_tW_{y,\tau}-F_\varepsilon(x,t,D W_{y,\tau},D^2W_{y,\tau})\geq1 \quad\text{in } Q_1.
$$
\end{lemma}

\begin{proof}
For $\tau\geq-1$ and $y\in\partial B_1$, we can construct
$$
W_{y,\tau}(x,t)=\frac{(t-\tau)^2}{2}+2W_y(x),
$$
which is a desired function apparently. Here $W_y(x)$ comes from Lemma \ref{lem6-1}. If $\tau=-1$ and $y\in B_1$, define
$$
W_{y,\tau}(x,t)=B(t+1)+|x-y|^\iota
$$
with
$$
\iota=\max\left\{2,\frac{q+2}{q+1},\frac{s+2}{s+1}\right\}.
$$
By virtue of the definition of $\iota$, we now evaluate
\begin{align*}
&\quad\partial_tW_{y,\tau}-F_\varepsilon(x,t,D W_{y,\tau},D^2W_{y,\tau})\\
&=B-\left[(\iota^2|x-y|^{2(\iota-1)}+\varepsilon^2)^\frac{q}{2}+a(x,t)(\iota^2|x-y|^{2(\iota-1)}+\varepsilon^2)^\frac{s}{2}\right]\\
&\quad\cdot\left((n-1)\iota+\iota(\iota-1)+(p-2)
\frac{\iota^2|x-y|^{2(\iota-1)}}{\iota^2|x-y|^{2(\iota-1)}+\varepsilon^2}\iota(\iota-1)\right)|x-y|^{\iota-2}\\
&\geq B-\left[(\iota^2|x-y|^{2(\iota-1)}+\varepsilon^2)^\frac{q}{2}+a^+(\iota^2|x-y|^{2(\iota-1)}+\varepsilon^2)^\frac{s}{2}\right]
\iota(n-1+p(\iota-1))|x-y|^{\iota-2}\\
&\geq\begin{cases}B-\big[(16+1)^\frac{q}{2}+a^+(16+1)^\frac{s}{2}\big]2(n-1+p) &{ \text{if } q\geq 0},\\[2mm]
B-\big[\iota^q|x-y|^{q(\iota-1)+\iota-2}+a^+\iota^s|x-y|^{s(\iota-1)+\iota-2}\big]\iota(n-1+p(\iota-1)) &{ \text{if } -1<s<0},\\[2mm]
B-\big[\iota^q|x-y|^{q(\iota-1)+\iota-2}+a^+(\iota^22^{2(\iota-1)}+1)^\frac{s}{2}\big]\iota(n-1+p(\iota-1)) &{{} \text{others}},
\end{cases}
\end{align*}
which leads to
$$
\partial_tW_{y,\tau}-F_\varepsilon(x,t,D W_{y,\tau},D^2W_{y,\tau})\geq B-C,
$$
where $C>0$ depends on $n,p,q,s$ and $a^+$. From that, we can see that $W_{y,\tau}$ will be a desired function as well, if we pick $B=C+1$.
\end{proof}

Recall that $a\vee b = \max\{a, b\}$ and $a\wedge b = \min\{a, b\}$ with $a$ and $b$ being two real numbers. With Lemma \ref{lem6-2} in hand, we can deduce the following result. The proof is the same as that of Theorem A.3 in \cite{IJS19}, so we omit it here.

\begin{corollary}
\label{cor6-3}
Suppose that $u\in C(\overline{Q_1})$ is a solution to \eqref{2-1} with $\varepsilon\in(0,1)$. Let $\varphi:=u\mid_{\partial_pQ_1}$ and $\rho$ be a modulus of continuity of $\varphi$. Then there is another modulus of continuity $\tilde{\rho}$, which depends on $n,p,q,s,a^+$ and $\rho$, such that, for any $(x,t)\in\overline{Q_1}$ and $(y,s)\in\partial_pQ_1$,
$$
|u(x,t)-u(y,s)|\leq \tilde{\rho}(|x-y|\vee\sqrt{|t-s|})
$$
holds true.
\end{corollary}

As a consequence, merging Lemmas \ref{lem2-1} and \ref{lem2-2} with Corollary \ref{cor6-3}, we could arrive at the boundary estimates (Proposition \ref{pro3-7}) by following the proof of Proposition 2.5 in \cite{JS17}. We omit the detailed proof.

\section*{Acknowledgements}

This work was supported by the National Natural Science Foundation of China (No. 12071098).




\begin{thebibliography}{[a]}

\bibitem{ACP11} R. Argiolas, F. Charro and I. Peral, On the Aleksandrov-Bakelman-Pucci estimate for some elliptic and parabolic nonlinear operators, Arch. Ration. Mech. Anal. \textbf{202} (2011) 875--917.

\bibitem{Att20} A. Attouchi, Local regularity for quasi-linear parabolic equations in non-divergence form, Nonlinear Anal. \textbf{199} (2020) 112051.

\bibitem{AP18} A. Attouchi and M. Parviainen, H\"{o}lder regularity for the gradient of the inhomogeneous parabolic normalized $p$-Laplacian, Commun. Contemp. Math. \textbf{20} (4) (2018), 27 pp.

\bibitem{APR17} A. Attouchi, M. Parviainen and E. Ruosteenoja, $C^{1,\alpha}$ regularity for the normalized $p$-Poisson problem, J. Math. Pures Appl. \textbf{108} (2017) 553--591.

\bibitem{AR18} A. Attouchi and E. Ruosteenoja, Remarks on regularity for $p$-Laplacian type equations in non-divergence form, J. Differential Equations \textbf{265} (2018) 1922--1961.

\bibitem{AR20} A. Attouchi and E. Ruosteenoja, Gradient regularity for a singular parabolic equation in non-divergence form, Discrete Contin. Dyn. Syst. \textbf{40} (10) (2020) 5955--5972.

\bibitem{BBO20} S. Baasandorj, S. S. Byun and J. Oh, Calder\'{o}n-Zygmund estimates for generalized double phase problems, J. Funct. Anal. \textbf{279} (7) (2020), 108670, 57 pp.

\bibitem{BG13} A. Banerjee and  N. Garofalo, Gradient bounds and monotonicity of the energy for some nonlinear singular diffusion equations, Indiana Univ. Math. J. \textbf{62} (2) (2013) 699--736.

\bibitem{BG2015} A. Banerjee and N. Garofalo, On the Dirichlet boundary value problem for the normalized $p$-Laplacian evolution, Commun. Pure Appl. Anal. \textbf{14} (1) (2015) 1--21.

\bibitem{BM19} A. Banerjee and I. H. Munive, Gradient continuity estimates for the normalized $p$-poisson equation, Commun. Contemp. Math.  \textbf{22} (8) (2020), 24 pp.



\bibitem{BCM18} P. Baroni, M. Colombo and G. Mingione, Regularity for general functionals with double phase, Calc. Var. Partial Differential Equations \textbf{57} (2018), Art. 62.

\bibitem{BD10} I. Birindelli and F. Demengel, Regularity and uniqueness of the first eigenfunction for singular fully nonlinear operators, J. Differential Equations \textbf{249} (2010) 1089--1110.

\bibitem{BS20} K. O. Buryachenko and I. I. Skrypnik, Local continuity and Harnack's inequality for double-Phase parabolic equations, Potential Anal, https://doi.org/10.1007/s11118-020-09879-9.

\bibitem{BO17} S. S. Byun and J. Oh, Global gradient estimates for non-uniformly elliptic equations, Calc. Var. Partial Differential Equations \textbf{56} (2017), Art. 36.

\bibitem{CGG91} Y. Chen, Y. Giga and S. Goto, Uniqueness and existence of viscosity solutions of generalized mean curvature flow equations, J. Differential Geom. \textbf{33} (1991) 749--786.

\bibitem{CDeF20} I. Chlebicka and C. De Filippis, Removable sets in non-uniformly elliptic problems, Ann. Mat. Pura Appl. \textbf{199} (2020) 619--649.

\bibitem{CZ20} I. Chlebicka and A. Zatorska-Goldstein, Generalized superharmonic functions with strongly nonlinear operator, Potential Anal., https://doi.org/10.1007/s11118-021-09920-5.

\bibitem{CS16} F. Colasuonno and M. Squassina, Eigenvalues for double phase variational integrals, Ann. Mat. Pura Appl. \textbf{195} (2016) 1917--1959.

\bibitem{CM15} M. Colombo and G. Mingione, Regularity for double phase variational problems, Arch. Ration. Mech. Anal. \textbf{215} (2015) 443--496.

\bibitem{CM215} M. Colombo and G. Mingione, Bounded minimisers of double phase variational integrals, Arch. Ration. Mech. Anal. \textbf{218} (2015) 219--273.

\bibitem{CM16} M. Colombo and G. Mingione, Calder\'{o}n-Zygmund estimates and non-uniformly elliptic operators, J. Funct. Anal. \textbf{270} (2016) 1416--1478.

\bibitem{Cra97} M. G. Crandall, Viscosity Solutions: A Primer. Viscosity Solutions and Applications (Montecatini Terme, 1995), in: Lecture Notes in Math, vol. 1660, Springer, Berlin, 1997, pp. 1--43.

\bibitem{CIL92} M. G. Crandall, H. Ishii and P. L. Lions, User's guide to viscosity solutions of second order partial differential equations, Bull. Amer. Math. Soc. \textbf{27} (1) (1992) 1--67.

\bibitem{DeF20} C. De Filippis, Gradient bounds for solutions to irregular parabolic equations with $(p,q)$-growth,  Calc. Var. Partial Differential Equations \textbf{59} (2020), Art. 171.

\bibitem{DeFM} C. De Filippis and G. Mingione, A borderline case of Calder\'{o}n-Zygmund estimates for non-uniformly elliptic problems, St. Petersburg Math. J. \textbf{31} (3) (2020)  455--477.

\bibitem{DeFM202} C. De Filippis and G. Mingione, Manifold constrained non-uniformly elliptic problems, J. Geom. Anal.  \textbf{30} (2) (2020) 1661--1723.

\bibitem{Dem11} F. Demengel, Existence's results for parabolic problems related to fully nonlinear operators degenerate or singular, Potential Anal. \textbf{35} (2011) 1--38.

\bibitem{DF85} E. DiBenedetto and A. Friedman, H\"{o}lder estimates for nonlinear degenerate parabolic systems, J. Reine Angew. Math. \textbf{357} (1985) 1--22.

\bibitem{Doe11} K. Does, An evolution equation involving the normalized $p$-Laplacian, Commun. Pure Appl. Anal. \textbf{10} (1) (2011) 361--396.

\bibitem{FVZZ20} Y. Fang, V. R\u{a}dulescu, C. Zhang and X. Zhang, Gradient estimates for multi-phase problems in Campanato spaces, Indiana Univ. Math. J., to appear.

\bibitem{FZ20} Y. Fang and C. Zhang, Equivalence between distributional and viscosity solutions for the double-phase equation, Adv. Calc. Var., https://doi.org/10.1515/acv-2020-0059.

\bibitem{GGIS91} Y. Giga, S. Goto, H. Ishii and M. H. Sato, Comparison principle and convexity preserving properties for singular degenerate parabolic equations on unbounded domains, Indiana Univ. Math. J. \textbf{40} (1991) 443--470.

\bibitem{HJ85} R. A. Horn and C. R. Johnson, Matrix analysis, Cambridge University Press, Cambridge, 1985.

\bibitem{IJS19} C. Imbert, T. Jin and L. Silvestre, H\"{o}lder gradient estimates for a class of singular or degenerate parabolic equations, Adv. Nonlinear Anal. \textbf{8} (2019) 845--867.

\bibitem{IS13} C. Imbert and L. Silvestre, $C^{1,\alpha}$ regularity of solutions of some degenerate fully non-linear elliptic equations, Adv. Math. \textbf{233} (2013) 196--206.

\bibitem{JS17} T. Jin and L. Silvestre, H\"{o}lder gradient estimates for parabolic homogeneous $p$-Laplacian equations, J. Math. Pures Appl. \textbf{108} (1) (2017) 63--87.

\bibitem{Juu14} P. Juutinen, Decay estimates in the supremum norm for the solutions to a nonlinear evolution equation, Proc. Roy. Soc. Edinburgh Sect. A \textbf{144} (3) (2014) 557--566.

\bibitem{JLM01} P. Juutinen, P. Lindqvist and J. J. Manfredi, On the equivalence of viscosity solutions and weak solutions for a quasi-linear equation, SIAM J. Math. Anal. \textbf{33} (3) (2001) 699--717.

\bibitem{JLP10} P. Juutinen, T. Lukkari and M. Parviainen, Equivalence of viscosity and weak solutions for the $p(x)$-Laplacian, Ann. Inst. H. Poincar\'{e} Anal. Non Lin\'{e}aire \textbf{27} (2010) 1471--1487.

\bibitem{KKK14} B. Kawohl, S. Kr\"{o}mer and J. Kurtz, Radial eigenfunctions for the game-theoretic $p$-Laplacian on a ball, Differential Integral Equations \textbf{27} (7-8) (2014) 659--670.

\bibitem{LSU68} O. A. Lady\v{z}enskaja, V.A. Solonnikov and N. N. Ural'ceva, Linear and Quasilinear Equations of Parabolic Type, Transl. Math. Monogr. 23, American Mathematical Society, Providence, 1968.

\bibitem{LM14} M. Lewicka and J. J. Manfredi, Game theoretical methods in PDEs, Boll. Unione Mat. Ital. \textbf{7} (3) (2014) 211--216.

\bibitem{LD18} W. Liu and G. Dai, Existence and multiplicity results for double phase problem, J. Differential Equations \textbf{265} (9) (2018) 4311--4334.

\bibitem{LPS13} H. Luiro, M. Parviainen and E. Saksman, Harnack inequality for $p$-harmonic functions via stochastic games, Comm. Partial Differential Equations \textbf{38} (11) (2013) 1985--2003.

\bibitem{MPR10} J. J. Manfredi, M. Parviainen and J. D. Rossi, An asymptotic mean value characterization for a class of nonlinear parabolic equations related to tug-of-war games, SIAM J. Math. Anal. \textbf{42} (5) (2010) 2058--2081.

\bibitem{MPR12} J. J. Manfredi, M. Parviainen and J. D. Rossi, Dynamic programming principle for tug-of-war games with noise, ESAIM Control Optim. Calc. Var. \textbf{18} (1) (2012) 81--90.

\bibitem{Mar89} P. Marcellini, Regularity of minimizers of integrals of the calculus of variations with non standard growth conditions, Arch. Ration. Mech. Anal. \textbf{105} (1989) 267--284.

\bibitem{Mar91} P. Marcellini, Regularity and existence of solutions of elliptic equations with $p,q$-growth conditions, J. Differential Equations \textbf{90} (1991) 1--30.

\bibitem{Mar96} P. Marcellini, Everywhere regularity for a class of elliptic systems without growth conditions, Ann. Sc. Norm. Super. Pisa Cl. Sci. \textbf{23} (4) (1996) 1--25.

\bibitem{OS97} M. Ohnuma and K. Sato, Singular degenerate parabolic equations with applications to the $p$-Laplace diffusion equation, Comm. Partial Differential Equations \textbf{22} (1997) 381--411.

\bibitem{PV} M. Parviainen and J. L. V\'{a}zquez, Equivalence between radial solutions of different parabolic gradient-diffusion equations and applications, Ann. Sc. Norm. Super. Pisa Cl. Sci., to appear.

\bibitem{PS09} Y. Peres, O. Schramm, S. Sheffield and D. B. Wilson, Tug-of-war and the infinity Laplacian, J. Amer. Math. Soc. \textbf{22} (2009) 167--210.

\bibitem{PS08} Y. Peres and S. Sheffield, Tug-of-war with noise: a game-theoretic view of the $p$-Laplacian, Duke Math. J. \textbf{145} (1) (2008) 91--120.

\bibitem{Rossi11} J. D. Rossi, Tug-of-war games and PDEs, Proc. Roy. Soc. Edinburgh Sect. A \textbf{141} (2) (2011) 319--369.

\bibitem{Ruo16} E. Ruosteenoja, Local regularity results for value functions of tug-of-war with noise and running payoff, Adv. Calc. Var. \textbf{9} (1) (2016) 1--17.

\bibitem{Wang13} Y. Wang, Small perturbation solutions for parabolic equations, Indiana Univ. Math. J. \textbf{62} (2) (2013) 671--697.

\bibitem{Zhi93} V. V. Zhikov, Lavrentiev phenomenon and homogeneization of some variational problems, C. R. Acad. Sci. Paris S\'{e}r I Math. \textbf{316} (1993) 435--439.

\bibitem{Zhi95} V. V. Zhikov, On Lavrentiev phenomenon, Russian J. Math. Phys. \textbf{3} (1995) 249--269.


\end{thebibliography}
\end{document}